\documentclass[11pt,a4paper,reqno]{amsart}
\usepackage[T1]{fontenc}
\usepackage{mathrsfs}
\usepackage{amsfonts,amssymb,amsmath,amsgen,amsthm}
\usepackage{color}
\usepackage{graphicx,color,epsfig}%les deux lignes a rajouter pour le graphique
\usepackage{amsbsy}
\usepackage{mathrsfs}
\usepackage{bbm}
\usepackage{titletoc}
\usepackage{dsfont}
\usepackage{mathabx}
\usepackage[all]{xy}

\theoremstyle{plain}
\newtheorem{theorem}{Theorem}[section]

\newtheorem{assumption}[theorem]{Assumption}
\newtheorem{lemma}[theorem]{Lemma}
\newtheorem{corollary}[theorem]{Corollary}
\newtheorem{proposition}[theorem]{Proposition}

\newtheorem{remark}[theorem]{Remark}

\numberwithin{equation}{section}

\catcode`@=11
\def\section{\@startsection{section}{1}%
  \z@{1.5\linespacing\@plus\linespacing}{.5\linespacing}%
  {\normalfont\bfseries\large\centering}}
\catcode`@=12
\newcommand{\Real}{{\mathbb{R}}}
\newcommand{\beq}{\begin{equation}}
\newcommand{\eeq}{\end{equation}}
\newcommand{\ham}{\mathcal{H}}

\newcommand{\beqa}{\begin{eqnarray}}
\newcommand{\eeqa}{\end{eqnarray}}

\numberwithin{equation}{section}
\def\ds{\displaystyle}
\def\ni{\noindent}
\def\bs{\bigskip}
\def\ms{\medskip}

\def\fref#1{{\rm (\ref{#1})}}

\def\pa{\partial}

\def\RR{\mathbb{R}}

\def\TT{\mathbb{T}}
\def\calJ{\mathcal J}
\def\calH{\mathcal H}

\begin{document}
\title{Nonlinear stability criteria for the HMF Model}
\author[M. Lemou]{Mohammed Lemou}
\address{CNRS, IRMAR,  Universit\'e de Rennes 1 and INRIA, IPSO Project}
\email{mohammed.lemou@univ-rennes1.fr}
\author[A. M. Luz]{Ana Maria Luz}
\address{IRMAR, Universit\'e de Rennes 1 and IME, Universidade Federal Fluminense}
\email{anamaria.luz@gmail.com}
\author[F. M\'ehats]{Florian M\'ehats}
\address{IRMAR, Universit\'e de Rennes 1 and INRIA, IPSO Project}
\email{florian.mehats@univ-rennes1.fr}
\begin{abstract}
We study the nonlinear stability of a large class of inhomogeneous steady state solutions to the Hamiltonian Mean Field (HMF) model. Under a simple criterion, we prove the nonlinear stability of steady states which are decreasing functions of the microscopic energy. To achieve this task, we extend to this context the strategy based on generalized rearrangement techniques which was developed recently for the gravitational Vlasov-Poisson equation. Explicit stability inequalities are established and our analysis is able to treat non compactly supported steady states to HMF, which are physically relevant in this context but induces additional difficulties, compared to the Vlasov-Poisson system.
\end{abstract}
\date{} % Activate to display a given date or no date (if empty),
         % otherwise the current date is printed 
\thanks{The authors acknowledge support by the ANR project Moonrise ANR-14-CE23-0007-01.  The work of A. M. Luz was supported by the Brazilian National Council for
Scientific and Technological Development (CNPq) under the program "Science without Borders" 249279/2013-4.}
\maketitle

\section{Introduction and main result}
\subsection{The HMF model}
In this paper, we are interested in the nonlinear stability of a class of inhomogeneous steady state solutions to the Hamiltonian mean-field (HMF) model \cite{messer-spohn,antoni-ruffo}. The HMF system is a kinetic model describing particles moving on a unit circle interacting via an infinite range attractive cosine potential. This model has been used as a toy-model of the Vlasov-Poisson system in the physical community, for the study of non equilibrium phase transitions \cite{chavanis1,chavanis2,antoniazzi,ogawa2}, of travelling clusters \cite{yamaguchi1,yamaguchi2} or of relaxation processes \cite{barre1,barre2,chavanis4}. The dynamics of perturbations of inhomogeneous steady states of the HMF model has been investigated in \cite{barre3,barre4} and the formal linear stability of steady states has been studied in \cite{chavanis3,OGW,yamaguchi3}. In particular, a simple criterion of linear stability has been derived in \cite{OGW}. Our aim here is to prove the nonlinear stability of inhomogeneous steady states under the same criterion, by adapting the techniques developed in \cite{ML,ML2} for the 3D Vlasov-Poisson system. 
The long-time validity of the N-particle approximation for the HMF model has been investigated in \cite{rousset1,rousset2} and the Landau-damping phenomenon near spatially homogeneous state has been studied recently in \cite{faou-rousset}.

In the HMF model, the distribution function of particles $f(t,\theta,v)$ solves the initial-valued problem
\begin{align}\label{sis1}
&\partial_t f+v\partial_{\theta} f-\partial_{\theta} \phi_f\partial_v f=0,\qquad(t,\theta,v)\in\Real_{+}\times\TT\times\Real,\\
&f(0,\theta,v)=f_{\rm init}(\theta,v)\geq 0,\nonumber
\end{align}
where $\TT$ is the flat torus $[0,2\pi]$ and where the self-consistent potential $\phi_f$ associated to a distribution function $f$ is defined by
\begin{equation}
\label{phi}
 \phi_f(\theta)=-\int^{2 \pi}_{0}\rho_f(\theta')\cos(\theta-\theta')d\theta', \qquad \rho_f(\theta)=\int_{\Real}f(\theta,v)dv.
\end{equation}
The so-called magnetization is the two-dimensional vector defined by
\begin{equation}
M_f=\int^{2 \pi}_{0}\rho_f(\theta)u(\theta) d\theta, \qquad \mbox{with}\quad u(\theta)=(\cos \theta,\sin \theta)^T
\end{equation}
and we have
\begin{equation}
 \phi_f(\theta)=-M_f\cdot u(\theta).
\end{equation}
The following quantities are invariant during the evolution:
\begin{itemize}
\item[--] the Casimir functions
\begin{equation}
\iint\ G( f(\theta,v))d\theta dv
\end{equation}
for any function $G\in\mathcal C^1(\Real_+)$ such that $G(0)=0$;
\item[--] the nonlinear energy
\begin{eqnarray}\label{ham}
\mathcal{H}(f)&=&\frac{1}{2}\iint\ v^2 f(\theta,v)d\theta dv +\frac{1}{2}\int \rho_f(\theta)\phi_f(\theta)d\theta \nonumber \\
&=&\frac{1}{2}\iint\ v^2 f(\theta,v) d\theta dv -\frac{1}{2}M_f\cdot\int \rho_f(\theta)u(\theta)d\theta \nonumber\\
&=&\frac{1}{2}\iint\ v^2 f(\theta,v) d\theta dv -\frac{1}{2}|M_f|^2;
\end{eqnarray}
\item[--] the total momentum
\begin{equation}
\iint\ vf(\theta,v)d\theta dv.
\end{equation}
\end{itemize}
Moreover, the HMF system enjoys the Galilean invariance, that is, if $f(t,\theta,v)$ is a solution, then so is $f(t,\theta+v_0 t,v+v_0)$, for $v_0\in \Real$.

\subsection{Statement of the main result}
We consider a stationary state of the form
\begin{equation}\label{Qexp}
f_0(\theta,v)=F(e_0(\theta,v)),\qquad \mbox{with}\quad e_0(\theta,v)=\frac{v^2}{2}+\phi_0(\theta),
\end{equation}
and where the potential associated to $f_0$ according to \eqref{phi} takes the form
$$\phi_0(\theta)=-m_0\cos \theta,\qquad \mbox{with }m_0> 0.$$
Here $F$ is a given function satisfying the following assumption.
\begin{assumption}
\label{hypF}
The function $F$ is a $\mathcal C^0$ function on $\RR$ satisfying the following properties. It is a $\mathcal C^1$ function on $(-\infty,e_*)$, for some $e_* \in \RR\cup \{+\infty\}$, with $F'<0$ on this interval. We also assume that $F(e)=0$ for $e\geq e_*$ when $e_*$ is finite, and that $\lim_{e\to +\infty}F(e)=0$ if $e_*=+\infty$. We denote by $F^{-1}$ its inverse function, which is a $\mathcal C^1$ function defined from $(0,\sup F)$ onto $(-\infty,e^*)$. The function $f_0$ given by \eqref{Qexp} is supposed to belong to the energy space $L^1((1+|v|^2)d\theta dv)$. Moreover, in the case $e_*<+\infty$ and $m_0=e_*$, we assume further that $\int_{-m_0}^{m_0}\log(m_0-e)F'(e)de<+\infty$.\\
%(ii) In the case $e_*=+\infty$, we assume further that $\|f_0\|_{L^p}<+\infty$, for some $0<p<1/3$.
\end{assumption}
\ni
{\bf Examples.} All the following typical examples that can be found in the literature fulfill our Assumption \ref{hypF}:
\begin{itemize}
 \item[(i)] Maxwell-Boltzmann distributions \cite{chavanis4}, $F(e)=A\exp(-\beta e)$.
 \item[(ii)] Polytropic distributions with compact support \cite{chavanis3}, $F(e)=A(e_*-e)_+^{\frac{1}{q-1}}$ with $q>1$.
 \item[(iii)] Polytropic distributions with non compact support \cite{chavanis3}, $F(e)=A(e_0+e)_+^{\frac{1}{q-1}}$ with $\frac{1}{3}<q<1$.
\item[(iv)] Lynden-Bell distributions \cite{chavanis1}, $F(e)=\frac{A}{1+B\exp(\beta e)}$.
\end{itemize}
\begin{remark}
Note that Assumption \ref{hypF} implies in particular that $f_0\in L^{\infty}$ since
$$\|f_0\|_{L^{\infty}}\leq F(-m_0).$$ It is also clear that $e_*$ is finite if and only if $f_0$ is compactly supported. We finally note that we  must have $e_*>-m_0$, otherwise $f_0=0$ and this contradicts the assumption $m_0>0$.
\end{remark}
\bs

Our aim is to prove the orbital stability of such steady state under the following criterion.
\begin{assumption}[Nonlinear stability criterion] We will assume that $f_0$ satisfies the following criterion
\label{criterion}
$$\kappa_0<1,$$
with
\begin{equation}
\label{kappa0}
\kappa_0=\int^{2 \pi}_{0}\!\!\int^{+\infty}_{-\infty} \left|F'\!\left(e_0(\theta,v)\right)\right|\left(\frac{\ds \int_{\mathcal D}(\cos \theta-\cos \theta')(e_0(\theta,v)-\phi_0(\theta'))^{-1/2}d\theta'}{\ds\int_{\mathcal D}(e_0(\theta,v)-\phi_0(\theta'))^{-1/2}d\theta'} \right)^2d\theta dv,
\end{equation}
where $${\mathcal D}=\left\{\theta'\in \TT\,:\,\,\phi_0(\theta')<e_0(\theta,v)\right\}.$$
\end{assumption}
\begin{remark}
Direct computations show that our criterion $\kappa_0<1$ is the same as the one derived in \cite{OGW}, that is
\begin{align*}
0<1&+\iint F'(e_0(\theta,v))\cos^2 \theta d\theta dv-\frac{4}{\sqrt{m_0}}\int_{-m_0}^{m_0}K(k(e))\left(\frac{2E(k(e))}{K(k(e))}-1\right)^2F'(e)de\\
&-\frac{4}{\sqrt{m_0}}\int_{m_0}^{+\infty}\frac{K(1/k(e))}{k(e)}\left(\frac{2k(e)^2E(1/k(e))}{K(1/k(e))}+1-2k(e)^2\right)^2F'(e)de,
\end{align*}
with $k(e)=\left(\frac{e+m_0}{2m_0}\right)^{1/2}$ and where $K(k)$ and $E(k)$ are respectively the complete elliptic integrals of first and second kinds, see e.g. \cite{barre3}.
\end{remark}
Before stating our main result, we first recall the usual notion of rearrangement which we adapt here to functions defined on the domain $\TT\times\Real$. For any nonnegative function $f\in L^1(\TT\times \Real)$, we define its distribution function as
\begin{equation}
\label{muf}
\mu_f(s)=\left\thickvert\left\{(\theta,v) \in \TT\times \Real: f(\theta,v)>s\right\}\right\thickvert,\quad \mbox{for all }s\geq 0,
\end{equation}
where $|A|$ denotes the Lebesgue measure of a set $A$. Note that $\mu_f(0)$ may be infinite, but $\mu_f(s)$ is finite for $s>0$. Let $f^\sharp$ be the pseudo-inverse of the function $\mu_f$, defined by
$$f^\sharp(s)=\inf\left\{t\geq 0,\,\mu_f(t)\leq s\right\}=\sup\left\{t\geq 0,\,\mu_f(t)> s\right\},\quad \mbox{for all }s\geq 0$$
with, in particular, $f^\sharp(0)=\|f\|_{L^\infty}\in \Real \cup\{+\infty\}$ and $f^\sharp(+\infty)=0$.  
It is well known that  $\mu_f$ is right-continuous and that for or all $s\geq 0$, $t\geq 0$,
\begin{equation}
\label{equivalence}
f^\sharp(s)>t \quad\Longleftrightarrow\quad  \mu_f(t)>s.
\end{equation}Next, we define the rearrangement $f^*$ of $f$ by
$$f^*(\theta,v)=f^\sharp\left(\left\thickvert B(0,\sqrt{\theta^2+v^2})\cap\TT\times \RR\right\thickvert\right),$$
where $B(0,R)$ denotes the open ball in $\Real^2$ centered at 0 with radius $R$.
%\begin{remark}
%We will see that the assumption $\|f_0\|_{L^{p}}<+\infty$, for some $0<p<1/3$,  can be relaxed into an assumption in terms of $\mu_{f_0}$ which is 
%$$\int_0^{+\infty} \mu_{f_0}(s)^3ds <+\infty.$$
%\end{remark}

%The Jacobian $a_\phi$: we have
%\begin{equation}\label{jacob}
%a_\phi=\text{meas}\left\{(\theta,v)\in\TT\times\Real:\frac{v^2}{2}+\phi(\theta)<e\right\}=2\sqrt{2}\int_{0}^{2\pi}\sqrt{\left(e-\phi(\theta)\right)}_{+}d\theta
%\end{equation}
Our main result is the following theorem.
\begin{theorem}
\label{mainthm}
Let $f_0$ be  a steady state of the form \fref{Qexp} satisfying Assumptions \ref{hypF} and \ref{criterion}. There exists $\delta >0$ such that, for all $f\in L^1\left((1+|v|^2)d\theta dv\right)$ satisfying $|M_f- M_{f_0(\cdot-\theta_f)}| < \delta$, we have
%\begin{align}
%\label{stab-inequ}
% \|f- f_0(\cdot+\theta_f)\|_{L^1}^2 &\leq  C\left({\textcolor{white}\int}\hspace{-0.3cm} \ham(f)-\ham(f_0)+ C (\|f_0\|_{L^1}+\|f\|_{L^1}) \|f^*-f_0^*\|_{L^1} \right. \nonumber \\ &
%\hspace{0cm} + \left. C \iint  \textcolor{red} {\left(F^{-1}(f^*(\theta,v))+m_0\right)_+} |f^*-f_0^*| d\theta dv \right),
 %\end{align}

 \begin{align}
\label{stab-inequ}
 \|f- f_0(\cdot-\theta_f)\|_{L^1}^2 &\leq  C\left({\textcolor{white}\int}\hspace{-0.3cm} \ham(f)-\ham(f_0)+ C (1+\|f\|_{L^1}) \|f^*-f_0^*\|_{L^1} \right. \nonumber \\ &
\left. +C \int_0^{+\infty}s^2\left(f_0^\sharp(s)- f^\sharp(s)\right)_+ ds + C 
\int_0^{+\infty}\mu_{f_0}(s)^2\beta_{f^*,f_0^*}(s) ds\right) ,
 \end{align}
 where  $\beta_{f^*,f_0^*}(s) =\left\thickvert\left\{(\theta,v) \in \TT\times \Real: f^*(\theta,v)\leq s<f_0^*(\theta, v)\right\}\right\thickvert, \mbox{for all}\  s\geq 0,$  and where $C$  is a positive constant depending only on $f_0$. The parameter $\theta_f$ is  defined by
 $M_f= |M_f| (\cos\theta_f, \sin \theta_f)^T.$ 
 In particular, if $f_0$ is a compactly supported steady state, then \fref{stab-inequ} reduces to 
  \begin{equation}
\label{stab-inequ-compact}
 \|f- f_0(\cdot-\theta_f)\|_{L^1}^2\leq  C\left({\textcolor{white}\int}\hspace{-0.3cm} \ham(f)-\ham(f_0)+ C (1+\|f\|_{L^1}) \|f^*-f_0^*\|_{L^1}\right).
\end{equation}
\end{theorem}
The proof of this theorem is given in Section \ref{proof-thm} and uses several steps which are developed in the following sections.
In Section \ref{section2}, we introduce the generalized rearrangements with respect to the microscopic energy, which enable to define a reduced energy function depending on the magnetization vector only. In Section \ref{section3} we show that, under the stability criterion $\kappa_0<1$, the magnetization of the steady state is a strict local minimizer of this reduced energy function and, in Section \ref{section4}, we prove a functional inequality that enables to control $f-f_0$.

\subsection{Proof of the orbital stability of $f_0$}
In this subsection, we show how to derive a stability result for the HMF model directly from our main Theorem \ref{mainthm}. 
We distinguish the two cases: $e_*<+\infty$  and $e_*=+\infty$.

\bs
\ni
{\em Case 1: $e_*<+\infty$ }. In this case $f_0$ is compactly supported and we can apply
\fref{stab-inequ-compact}, that is we have
\begin{align}
\label{stab-inequ-fini}
 \|f- f_0(\cdot-\theta_f)\|_{L^1} ^2 \leq  C\left({\textcolor{white}\int}\hspace{-0.3cm} \ham(f)-\ham(f_0) + C ( {\textcolor{white}\int}\hspace{-0.2cm} 1+\|f\|_{L^1}) \|f^*-f_0^*\| \right)
 \end{align}
 for all $f$ satisfying $|M_f- M_{f_0(\cdot-\theta_{f})}|< \delta$. 
 Let $f_{init}\in L^1\left((1+|v|^2)dvd\theta\right)$ be any initial data   for the HMF equation \fref{sis1} such that $\theta_{f_{init}}=0$ and 
 $$ \|(1+|v|^2)(f_{init}-f_0)\|_{L^1}<\eta,$$ 
 where $0<\eta<\min(1,\delta/2)$ will be made precise later on. This implies in particular that
 \beq
 \label{MFL1}\left|M_{f_{init}} - M_{f_0} \right|\leq\|f_{init}-f_0\|_{L^1}< \eta< \delta/2,
 \eeq
 and then
 \begin{align*} 
 |\ham(f_{init})-\ham(f_0)| = &\left|\iint \left(\frac{v^2}{2}+\phi_0(\theta)\right)(f_{init}-f_0)d\theta dv - \frac{1}{2}\left|M_{f_{init}} - M_{f_0} \right|^2\right|\\
 & \leq (m_0+1)\eta.
 \end{align*}
 Now the contractivity property of the rearrangement implies that $\|f_{init}^*-f_0^*\|_{L^1} \leq  \|f_{init}-f_0\|_{L^1} <\eta$ and then
 \begin{align*}
 |\ham(f_{init})-\ham(f_0)| + C ( {\textcolor{white}\int}\hspace{-0.2cm}&1+\|f_{init}\|_{L^1}) \|f_{init}^*-f_0^*\| \\
 & \leq  \left[m_0+1+ C\left(2+\|f_0\|_{L^1}\right)\right]\eta.
 \end{align*}
 We then choose $\eta$ such that
\beq
\label{choixeta}
\eta <  \min\left(1, \delta/2,  \left[m_0+1+ C\left(2+ \|f_0\|_{L^1}\right)\right]^{-1} \delta^2/(4C)\right).
\eeq
Let now $f(t)$ be a solution to the HMF model with initial data $f_{init}$. {}From the conservation properties of this model, to wit 
$\calH(f(t)) \leq \calH(f_{init})$ and $f(t)^*= f_{init}^*,$
and from \fref{stab-inequ-fini} we then get
\begin{align*}
\|f(t)- f_0(\cdot-\theta_{f(t)})\|_{L^1} ^2& \leq C\left(\ham(f(t))-\ham(f_0)+ C ( {\textcolor{white}\int}\hspace{-0.2cm} 1+\|f(t)\|_{L^1}) \|f(t)^*-f_0^*\|\right)\\
& < \delta^2/4,  
\end{align*} 
as long as  $|M_{f(t)}- M_{f_0(\cdot-\theta_{f(t)})}| < \delta $.  
 In fact we shall prove that we have 
 \beq 
 \label{klamlam} |M_{f(t)}- M_{f_0(\cdot-\theta_{f(t)})}| < \delta, \ \ \forall t\geq 0.\eeq
 Indeed,  at $t=0$ we have $|M_{f(0)}- M_{f_0(\cdot-\theta_{f(0)}})| < \delta/2$ by assumption on $f_{init}$ (see \fref{MFL1}). If at some time $t$  we have $|M_{f(t)}- M_{f_0(\cdot-\theta_{f(t)}})| \geq  \delta$, then by continuity  in time there is some time $t_0$ such that $|M_{f(t_0)}- M_{f_0(\cdot-\theta_{f(t_0)}})| = 2\delta/3<\delta.$
We thus get 
$$\|f(t_0)- f_0(\cdot-\theta_{f(t_0)})\|_{L^1} < \delta/2, \ \ \forall t\geq 0. $$
But this implies
$$  2\delta/3= |M_{f(t_0)}- M_{f_0(\cdot-\theta_{f(t_0)})}| \leq \|f(t_0)- f_0(\cdot-\theta_{f(t_0)})\|_{L^1}  < \delta/2,$$
which is a contradiction, and claim \fref{klamlam} is proved.  We conclude from Theorem \ref{mainthm} that
\begin{equation}
\label{control}
\|f(t)- f_0(\cdot-\theta_{f(t)})\|_{L^1} ^2 \leq C\left(|\ham(f_{init})-\ham(f_0)| + C ( {\textcolor{white}\int}\hspace{-0.2cm}1+\|f_{init}\|_{L^1}) \|f_{init}^*-f_0^*\|\right),
\end{equation}
for all $t\geq 0$. The orbital stability of the solution $f(t)$ is then proved.

\bs
%\begin{remark}
%In fact, the shift $\theta_{f(t)}$ can be determined up to a small term. Indeed, we have
%\begin{align*}
%\theta_{f(t)}=\frac{1}{\|f_0\|_{L^1}}\iint\theta f_0(\theta-\theta_{f(t)},v)d\theta dv
%\end{align*}
%hence, denoting
%$$\widetilde \theta_{f(t)}=\frac{1}{\|f_0\|_{L^1}}\iint \theta f(t,\theta,v)d\theta dv,$$
%we obtain
%\begin{align*}
%\left|\theta_{f(t)}-\widetilde \theta_{f(t)}\right|\leq \frac{2\pi}{\|f_0\|_{L^1}}\|f(t_0)- f_0(\cdot-\theta_{f(t_0)})\|_{L^1},
%\end{align*}
%this difference being controlled thanks to \eqref{control}. It remains to remark that $\widetilde \theta_{f(t)}$ is known, since, by \eqref{sis1},
%$$\frac{d}{dt}\iint \theta f(t,\theta,v)d\theta dv=\iint v f(t,\theta,v)d\theta dv\quad\mbox{ and }\quad \frac{d}{dt}\iint v f(t,\theta,v)d\theta dv=0,$$
%so
%$$\widetilde \theta_{f(t)}\|f_0\|_{L^1}=t\iint v f_{init}(t,\theta,v)d\theta dv+\iint \theta f_{init}(t,\theta,v)d\theta dv.$$
%\end{remark}

\bs
\ni
{\em Case 2: $e_*=+\infty$ }.  In this case $f_0$ is not compactly supported and we shall
use inequality \fref{stab-inequ} of Theorem \ref{mainthm}.  
The quantity $\mu_{f_0}(s)$ involved in \fref{stab-inequ} is no longer bounded and presents a singularity at $s=0$. Therefore we shall  need to prove the following claim: if $\|f_{init}^{n}-f_0\|_{L^1} \to 0$ then 
\beq
\label{hadra}
\int_0^{+\infty}s^2\left(f_0^\sharp(s)- f_{init}^{n\sharp}(s)\right)_+ ds \to 0\quad \mbox{and} 
\quad \int_0^{+\infty}\mu_{f_0}(s)^2\beta_{f_{init}^{n*},f_0^*}(s) ds \to 0 ,\eeq
as $n\to +\infty$ up to the extraction of a subsequence. Once this claim is proved, the rest of the stability proof is exactly the same as in the case of a compactly supported steady state $f_0$.
Let us then prove claim \fref{hadra}.  We start by proving the first limit of this claim. Assume that $\|f_{init}^{n}-f_0\|_{L^1} \to 0$. Since $\|f_{init}^{n\sharp}-f_0^\sharp\|_{L^1(\RR^+)} \leq\|f_{init}^n-f_0\|_{L^1} \to 0$, we deduce that   $s^2\left(f_0^\sharp(s)- f_{init}^{n\sharp}(s)\right)_+ \to 0$ as $n\to 0$ for almost very $s\geq 0$, up to a substraction of a subsequence. But we have
 $$ s^2\left(f_0^\sharp(s)- f_{init}^{n\sharp}(s)\right)_+\leq s^2f_0^\sharp(s), \quad \mbox{and} \quad
 \int_0^{+\infty} s^2f_0^\sharp(s)<+\infty,$$
 (see \fref{fsharpenergy}  for the second inequality).
 Therefore, by dominated convergence we can pass to the limit inside the integral and get the first convergence in claim \fref{hadra}.
 Now we prove the second limit of claim \fref{hadra}. Assume again that $\|f_{init}^{n}-f_0\|_{L^1} \to 0$, then 
$$\int_0^{+\infty} \beta_{f_{init}^{n*},f_0^*}(s) ds = \iint (f_0^*- f_{init}^{n*})_+ d\theta dv \leq \|f_{init}^{n*}-f_0^*\|_{L^1}\leq 
\|f_{init}^{n}-f_0\|_{L^1} \to 0$$
 as $n\to +\infty$. This means that 
 $\beta_{f_{init}^{n*},f_0^*}(s) \to 0$   for almost every $s\geq 0$, up to an extraction of a subsquence. We then conclude that the quantity $\mu_{f_0}(s)^2\beta_{f_{init}^{n*},f_0^*}(s)$ arising in  \fref{hadra} converges to $0$ for almost every $s\geq 0$ (up to an extraction). Therefore, to end the proof of the second  limit in claim \fref{hadra}, it is sufficient to dominate this quantity by an $L^1$ function in $s\in (0,\|f_0\|_{L^\infty})$ uniformly in $n$.  To this purpose, we observe that
 $\beta_{f_{init}^{n*},f_0^*}(s) \leq \mu_{f_0}(s)$ and then 
 $$\mu_{f_0}(s)^2\beta_{f^{n*},f_0^*}(s) \leq \mu_{f_0}(s)^3, \qquad \forall\ s> 0.$$
To prove that  the rhs of this inequality  is integrable on $\RR_+$, we write
$$\int_0^{+\infty} s^2 f_0^\sharp (s) ds = \int_ 0^{+\infty} s^2 \left(\int_0^{f_0^\sharp (s)}dt\right) ds =\int_ 0^{+\infty}
 \left(\int_{f_0^\sharp (s)>t}s^2ds\right) dt$$
%  is finite by assumption \ref{hypF} on $f_0$ and is integrable with respect to $s\in (0,\|f_0\|_{L^\infty}),$
 and using \fref{equivalence} we get
 $$ \int_0^{+\infty} s^2 f_0^\sharp (s) ds = \int_ 0^{+\infty}
 \left(\int_{0\leq s< \mu_{f_0}(t)}s^2ds\right) dt= \frac{1}{3} \int_ 0^{+\infty}\mu_{f_0}(t)^3 dt.$$ 
 Since from \fref{fsharpenergy} we have $\int_0^{+\infty} s^2 f_0^\sharp (s) ds<+\infty$, the proof of 
 claim \fref{hadra} is complete.   This ends the proof of the orbital stability in all cases.
%proceed differently to prove the orbital stability of $f_0$. We use a contradiction argument. Assume orbital stability does not hold for $f_0$, then there exist $\eps_0>0$,  a sequence $f_{init}^n=f^n(0)$   converging to $f_0$ in the energy space
%$$ \|(1+|v|^2)(f_{init}^n-f_0)\|_{L^1}\rightarrow 0, \quad \mbox{as} \quad n\to +\infty,$$ 
%and a sequence $t_n>0$ such that for all $\tilde\theta\in \TT$
%$$ \|(1+|v|^2)(f^n(t_n,\cdot+\tilde \theta, \cdot)-f_0)\|_{L^1}>\eps_0.$$
%Here $f^n(t):=f^n(t,\theta,v)$ is the solution to \fref{sis1} at time $t$ with initial data $f_{init}^n=f^n(0)$.
% We have
% \beq
% \label{MFL1n}\left|M_{f_{init}^n} - M_{f_0} \right|\leq\|f_{init}^n-f_0\|_{L^1}\to 0,\quad  \mbox{as} \quad n\to \infty.
% \eeq
% and 
% \begin{align*} 
% |\ham(f_{init}^n)-\ham(f_0)| = &\left|\iint \left(\frac{v^2}{2}+\phi_0(\theta)\right)(f_{init}^n-f_0)d\theta dv - \frac{1}{2}\left|M_{f_{init}^n} - M_{f_0} \right|^2\right|\\
% & \leq (m_0+1)\left[ \|(1+|v|^2)(f_{init}^n-f_0)\|_{L^1}+ \left|M_{f_{init}^n} - M_{f_0} \right|^2\right].
% \end{align*}
% This implies 
% $$
% \ham(f_{init}^n)-\ham(f_0)\rightarrow 0 \quad  \mbox{as} \quad n\to \infty.
% $$
% Now from the contractivity property of the rearrangement $\|f_{init}^{n*}-f_0^*\|_{L^1} \leq  \|f_{init}^n-f_0\|_{L^1}\to 0$ , we also have
% $$|\ham(f_{init}^n)-\ham(f_0)| + C (\|f_0\|_{L^1}+\|f_{init}^n\|_{L^1}) \|f_{init}^{n*}-f_0^*\| \rightarrow 0 \quad  \mbox{as} \quad n\to \infty.$$

\qed

\section{The reduced energy functional}
\label{section2}
The aim of this section is to introduce a reduced energy functional $\mathcal J(|M_f|)$ which depends only on the modulus of the magnetization and which is such that  $\mathcal J(|M_f|)-\mathcal J(m_0)$ (recall that $M_{f_0}=(m_0,0)^T$ with $m_0\geq 0$) is controlled by the relative nonlinear energy $\mathcal H(f)-\mathcal H(f_0)$, up to conserved quantities.

\subsection{Generalized rearrangements with respect to the microscopic energy}

Our purpose now is to define a generalized symmetric nonincreasing rearrangement with respect to the microscopic energy $e=\frac{v^2}{2}+\phi(\theta)$, where the potential $\phi$ is a given $\mathcal C^\infty$ function on $\TT$. We introduce the quantity
\begin{equation}
\label{aphi}
a_\phi(e)=\left\thickvert\left\{(\theta,v) \in \TT\times \Real\,:\,\,\frac{v^2}{2}+\phi(\theta)<e \right\}\right\thickvert,\quad \mbox{for all }e\in \RR.
\end{equation}
It has the explicit expression
$$a_\phi(e)=2\sqrt{2}\int_0^{2\pi}\sqrt{(e-\phi(\theta))_+}\,d\theta,$$
with the usual notation $x_+=\max(0,x)$. It is readily seen that $a_\phi$ is continuous on $\RR$, vanishes on $(-\infty,\min \phi]$ and is strictly increasing from $[\min \phi,+\infty)$ to $[0,+\infty)$. This enables to define its inverse $a_\phi^{-1}$ on $[0,+\infty)$. Note that, for all $e\in \RR$,
\begin{equation}
\label{encadrement0}
4\pi\sqrt{2}(e-\max\phi)_+^{1/2}\leq a_\phi(e)\leq 4\pi\sqrt{2}(e-\min\phi)_+^{1/2},
\end{equation}
which implies, for all $s\in \RR_+$,
\begin{equation}
\label{encadrement}
\frac{s^2}{32\pi^2}+\min \phi\leq a_\phi^{-1}(s)\leq \frac{s^2}{32\pi^2}+\max \phi.
\end{equation}
We now introduce the generalized rearrangement with respect to the microscopic energy.
\begin{lemma}
\label{lem1}
Let $\phi\in \mathcal C^\infty(\TT)$ and let $a_\phi$ be the function defined by \eqref{aphi}. Let $f\in L^1(\TT\times \RR)$, nonnegative. Then the  function
$$f^{*\phi}(\theta,v)=f^\sharp\left(a_\phi\left(\frac{v^2}{2}+\phi(\theta)\right)\right),\qquad (\theta,v)\in\TT\times \RR$$
is equimeasurable to $f$, that is $\mu_{f^{*\phi}}=\mu_f$, where $\mu_f$ is defined by \eqref{muf}. In the sequel, the function $f^{*\phi}$ is called (decreasing) rearrangement with respect to the microscopic energy $\frac{v^2}{2}+\phi(\theta)$.
\end{lemma}
\begin{proof}
{}From the right continuity of $\mu_f$, we infer that, for all $s\geq 0$, $t\geq 0$,
\begin{equation}
\label{equivalence}
f^\sharp(s)>t \quad\Longleftrightarrow\quad  \mu_f(t)>s.
\end{equation}
Therefore, for all $t\geq 0$,
\begin{align*}
 \mu_{f^{*\phi}}(t)&=\left\thickvert\left\{(\theta,v) \in \TT\times \RR:\quad  f^\sharp\left(a_\phi\left(\frac{v^2}{2}+\phi(\theta)\right)\right)>t\right\}\right\thickvert\\
&=\left\thickvert\left\{(\theta,v) \in \TT\times \RR:\quad  a_\phi\left(\frac{v^2}{2}+\phi(\theta)\right)<\mu_f(t)\right\}\right\thickvert\\
&=\left\thickvert\left\{(\theta,v) \in \TT\times \RR:\quad  \frac{v^2}{2}+\phi(\theta)<a_\phi^{-1}(\mu_f(t))\right\}\right\thickvert\\
&=a_\phi\left(a_\phi^{-1}(\mu_f(t))\right)=\mu_f(t).
\end{align*}
\end{proof}

Finally, we state a technical lemma dealing with the case of potentials which have the special form of potentials of the HMF model. For $e\in \RR$, $m\in \RR_+^*$ and $\phi(\theta)=-m\cos \theta$ we denote
$$\alpha_m(e)=a_\phi(e)=2\sqrt{2}\int_0^{2\pi}\sqrt{(e+m \cos \theta)_+}\,d\theta=\sqrt{m}\,\alpha_1\left(\frac{e}{m}\right)$$
and introduce the angle
\begin{equation}
\label{defthetam}
\theta_m(e)=\left\{\begin{array}{ll}0,\qquad &\mbox{if }e\leq -m,\\
\mbox{arccos }(-e/m)\in (0,\pi),\qquad &\mbox{if }-m<e<m,\\
\pi,\qquad &\mbox{if }e\geq m.
\end{array}\right.
\end{equation}
The function $\alpha_1(e)$ and its derivative $\alpha'_1(e)$ are represented on Figure \ref{fig1}. 
\begin{figure}[!htbp]
  \centerline{
\includegraphics[width=.5\textwidth]{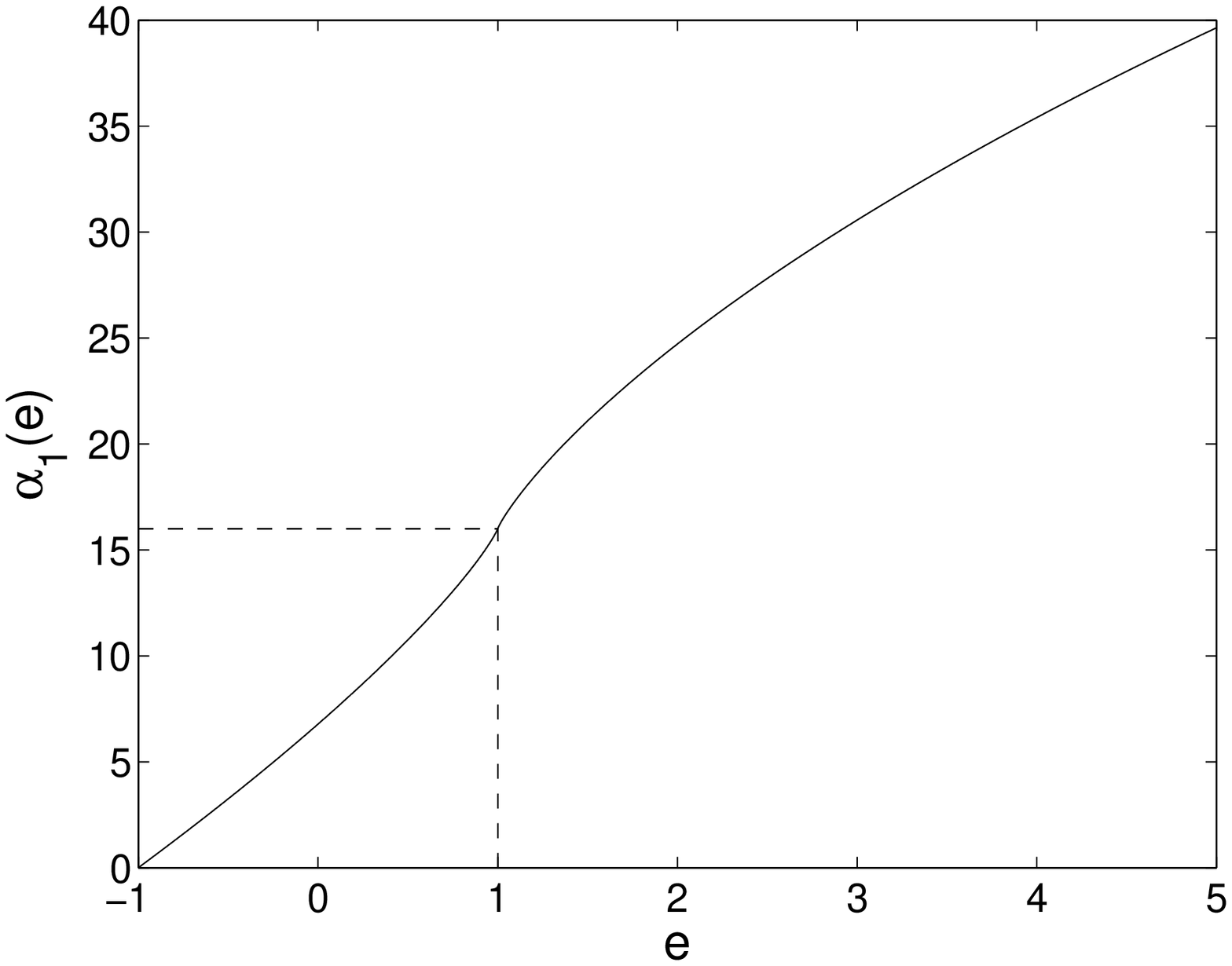}\hspace*{-8mm}
 \includegraphics[width=.5\textwidth]{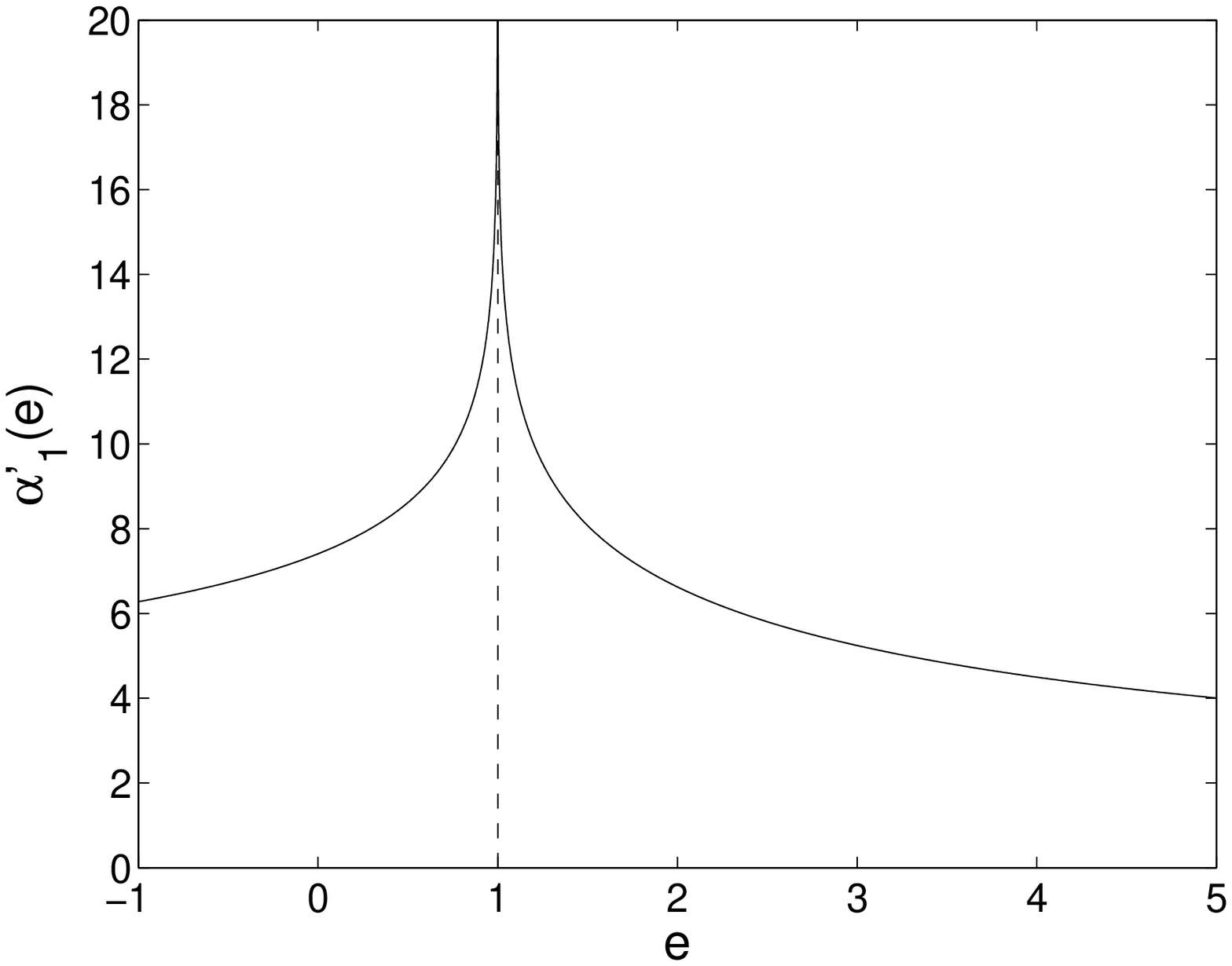}}
      \caption{Function $\alpha_1$ (left) and its derivative $\alpha'_1$ (right).}
\label{fig1}
\end{figure}
The proof of the following lemma is deffered to the Appendix.
\begin{lemma}[Properties of the function $\alpha_1$]
\label{lemtechnique}
Let 
\begin{equation}
\label{alpha1}
\alpha_1(e)=4\sqrt{2}\int_0^{\theta_1(e)}(e+\cos \theta)^{1/2}\,d\theta \qquad \mbox{for }e\in \RR.
\end{equation}
This function satisfies the following properties:\\
(i) $\alpha_1$ is a continuous nondecreasing function from $\RR$ to $\RR_+$ and $\alpha_1(e)=0$ for $e\leq -1$.\\
(ii) $\alpha_1$ is a strictly increasing and strictly convex $\mathcal C^1$ function on $[-1,1)$. Its derivative for $e\in(-1,1)$ is given by
\begin{equation}
\label{aphi'}
\alpha_1'(e)=2\sqrt{2}\int_0^{\theta_1(e)}(e+\cos \theta)^{-1/2}\,d\theta.
\end{equation}
and its right-derivative at $e=-1$ is equal to $2\pi$.\\
(iii) $\alpha_1$ is a strictly increasing and strictly concave $\mathcal C^1$ function on $(1,+\infty)$. Its derivative for $e\in(1,\infty)$ is still given by \eqref{aphi'} and we have
$$\alpha_1(e)\sim4\pi\sqrt{2e},\qquad \alpha'_1(e)\sim2\pi\sqrt{2/e}\quad \mbox{as}\quad e\to +\infty.$$
(iv) We have $\alpha_1(1)=16$ and
\begin{equation}
\label{equiv}
\alpha'_1(e)\sim -2\log |e-1|\quad \mbox{as}\quad e\to 1.
\end{equation}
(v) The inverse $\alpha_1^{-1}$ of the function $\alpha_1\,:[-1,+\infty)\mapsto [0,+\infty)$ is a strictly increasing $\mathcal C^1$ function, defined on $[0,+\infty)$, satisfying
$$(\alpha_1^{-1})'(s)=\frac{1}{\alpha_1'\circ \alpha_1^{-1}(s)}\quad \mbox{for}\quad s\in \RR_+\setminus\{0,16\},$$
with $\alpha_1'$ given by \eqref{aphi'}, and
$$(\alpha_1^{-1})'(0)=\frac{1}{2\pi},\qquad (\alpha_1^{-1})'(16)=0.$$
\end{lemma}

\subsection{Reduction to a functional of the magnetization vector}

In this subsection, we prove the following result.
\begin{proposition}
\label{prop1}
For all $f\in L^1((1+|v|^2)d\theta dv)$, we have
\begin{align}
\label{inegreduction}
\calJ(|M_f|)-\calJ(|M_{f_0}|)&\leq \calH(f)-\calH(f_0)+ (2m_0+3\|f\|_{L^1})\|f^*-f_0^*\|_{L^1}\\
&\quad +\int_{0}^{+\infty}
%\left|F^{-1}(f^\sharp(s))\right|
%\textcolor{red} {\left(F^{-1}(f^\sharp(s))+m_0\right)_+}\,\left|f_0^\sharp(s)-f^\sharp(s)\right| ds
 s^2\,\left(f_0^\sharp(s)-f^\sharp(s)\right)_+ ds
\nonumber 
\end{align}
where, for all $m\in \RR_+$,
\begin{equation}
\label{defJ}
\calJ(m)=\frac{m^2}{2}+\int_{-\infty}^{+\infty} \int^{2 \pi}_{0}\left(\frac{v^2}{2}+\phi \right)f_0^{*\phi}d\theta dv
\end{equation}
with $\phi(\theta)=-m\cos \theta$.
\end{proposition}
\begin{proof}
Writing the difference $\calH(f)-\calH(f_0)$ between the nonlinear energies as
\begin{align*}
\ham(f)-\ham(f_0)=& \iint\left(\frac{v^2}{2}+\phi_f \right)\left(f-f_0\right) d\theta dv-\frac{1}{2}\left(|M_f|^2-|M_{f_0}|^2\right)- \iint \phi_f (f-f_0)d\theta dv\\
=& \iint\left(\frac{v^2}{2}+\phi_f \right)\left(f-f^{*\phi_f}\right) d\theta dv+\iint\left(\frac{v^2}{2}+\phi_f \right)\left(f^{*\phi_f}-f_0^{*\phi_f}\right) d\theta dv\\
&+\iint\left(\frac{v^2}{2}+\phi_f \right)\left(f_0^{*\phi_f}-f_0\right) d\theta dv+\frac{1}{2}|M_f-M_{f_0}|^2\\
=& I_1+I_2+I_3+\frac{1}{2}|M_f-M_{f_0}|^2,
\end{align*}
we organize the proof in three steps.

\bs
\ni
{\em Step 1: Identification of $\calJ(|M_f|)-\calJ(|M_{f_0}|)$.}
\nopagebreak

\ms
\ni
Let us first prove that
\begin{equation}
\label{fixedpoint}
f_0=f_0^{*\phi_0},
\end{equation}
which amounts to proving that
\begin{equation}
\label{fixedpointbis}
F(e)=f_0^\sharp\circ a_{\phi_0}(e),\qquad \forall e\geq \min \phi_0.
\end{equation}
Recall that $a_{\phi_0}$ is invertible from $[\min \phi_0,+\infty)$ to $[0,+\infty)$ and denote $G=F\circ a_{\phi_0}^{-1}$ on $[0,+\infty)$. Recall also that $F$ is assumed to be continuously decreasing. Hence, so is the function $G$ and then it is standard that $G^\sharp=G$, see for instance \cite{rakotoson}. Now, for all $t\geq 0$,
\begin{align*}
\mu_{f_0}(t)&=\left\thickvert\left\{(\theta,v) \in \TT\times \Real\,:\,\,F\left(\frac{v^2}{2}+\phi_0(\theta)\right)>t \right\}\right\thickvert\\
&=\left\thickvert\left\{(\theta,v) \in \TT\times \Real\,:\,\,G\circ a_{\phi_0}\left(\frac{v^2}{2}+\phi_0(\theta)\right)>t \right\}\right\thickvert\\
&=\left\thickvert\left\{(\theta,v) \in \TT\times \Real\,:\,\,G^\sharp\circ a_{\phi_0}\left(\frac{v^2}{2}+\phi_0(\theta)\right)>t \right\}\right\thickvert.
\end{align*}
Hence, applying the \eqref{equivalence} to the function
$$\mu_G(s)=\left|\left\{t\geq 0:\, G(t)>s\right\}\right|,\quad \mbox{for all }s\geq 0,$$
and to its pseudo-inverse $G^\sharp$, we get
\begin{align*}
\mu_{f_0}(t)&=\left\thickvert\left\{(\theta,v) \in \TT\times \Real\,:\,\,a_{\phi_0}\left(\frac{v^2}{2}+\phi_0(\theta)\right)<\mu_{G}(t) \right\}\right\thickvert\\
&=\left\thickvert\left\{(\theta,v) \in \TT\times \Real\,:\,\,\frac{v^2}{2}+\phi_0(\theta)<a_{\phi_0}^{-1}(\mu_{G}(t)) \right\}\right\thickvert\\
 &=a_{\phi_0}\circ a_{\phi_0}^{-1}(\mu_{G}(t))=\mu_{G}(t).
 \end{align*}
 {}From this, we deduce that $f_0^\sharp=G^\sharp=G$, which gives \eqref{fixedpointbis} and ends the proof of \eqref{fixedpoint}.
 Consequently,
\begin{align*}
I_3&+\frac{1}{2}|M_f-M_{f_0}|^2\\
&=\iint\left(\frac{v^2}{2}+\phi_f \right)\left(f_0^{*\phi_f}-f_0\right) d\theta dv+\frac{1}{2}|M_f-M_{f_0}|^2\\
&=\iint\left(\frac{v^2}{2}+\phi_f \right)\left(f_0^{*\phi_f}-f_0^{*\phi_0}\right) d\theta dv+\frac{1}{2}|M_f-M_{f_0}|^2\\
&=\iint\left(\frac{v^2}{2}+\phi_f \right)f_0^{*\phi_f}d\theta dv-\iint\left(\frac{v^2}{2}+\phi_0\right)f_0^{*\phi_0} d\theta dv\\
&\qquad+\int\left(\phi_0-\phi_f\right)\rho_{f_0}\,d\theta+\frac{1}{2}|M_f-M_{f_0}|^2\\
&=\iint\left(\frac{v^2}{2}+\phi_f \right)f_0^{*\phi_f}d\theta dv+\frac{1}{2}|M_f|^2-\iint\left(\frac{v^2}{2}+\phi_0\right)f_0^{*\phi_0} d\theta dv-\frac{1}{2}|M_{f_0}|^2.
\end{align*} 
We observe now that $\phi_f$ can be written as $\phi_f(\theta)=- |M_f|\cos(\theta-\theta_M)$ for some $\theta_M\in\TT$. Hence, by periodicity, we have
$$\iint\left(\frac{v^2}{2}+\phi_f \right)f_0^{*\phi_f}d\theta dv+\frac{1}{2}|M_f|^2=\mathcal J(|M_f|),$$
where $\mathcal  J$ is defined by \fref{defJ}, 
and the same holds for $\phi_0$. We thus have
 $$I_3+\frac{1}{2}|M_f-M_{f_0}|^2=\mathcal J(|M_f|)-\mathcal J(|M_{f_0}|).$$

\bs
\ni
{\em Step 2: Positivity of $I_1$.}
\nopagebreak

\ms
\ni
We have, using Fubini,
\begin{align*}
I_1&=\iint\left(\frac{v^2}{2}+\phi_f \right)\left(f-f^{*\phi_f}\right) d\theta dv\\
&=\iint\left(\frac{v^2}{2}+\phi_f \right)\left(\int_0^{f}dt-\int_0^{f^{*\phi_f}}dt\right) d\theta dv\\
&=\int_0^{+\infty}\left(\iint_{f>t}\left(\frac{v^2}{2}+\phi_f \right)d\theta dv-\iint_{f^{*\phi_f}>t}\left(\frac{v^2}{2}+\phi_f \right)d\theta dv\right) dt\\
&=\int_0^{+\infty}\left(\iint_{A(t)}\left(\frac{v^2}{2}+\phi_f \right)d\theta dv-\iint_{B(t)}\left(\frac{v^2}{2}+\phi_f \right)d\theta dv\right) dt
\end{align*}
where, for all $t\geq 0$, we have denoted
\begin{align*}
A(t)&=\left\{(\theta,v) \in \TT\times \Real\,:\,f^{*\phi_f}(\theta,v)\leq t<f(\theta,v)\right\},\\
B(t)&=\left\{(\theta,v) \in \TT\times \Real\,:\,f(\theta,v)\leq t<f^{*\phi_f}(\theta,v)\right\}.
\end{align*}
Since $f^{*\phi_f}$ is a decreasing function of $\frac{v^2}{2}+\phi_f$, we clearly have
$$\forall (\theta,v)\in A(t),\quad \forall (\theta',v')\in B(t),\qquad \frac{v^2}{2}+\phi_f(\theta)>\frac{v'^2}{2}+\phi_f(\theta').$$
Moreover, from the equimeasurability of $f$ and $f^{*\phi_f}$, we have $|A(t)|=|B(t)|$.
Consequently, we obtain $I_1\geq 0$.

\bs
\ni
{\em Step 3: Control of $|I_2|$ by $\|f^*-f_0^*\|_{L^1}$.}
\nopagebreak

\ms
\ni
Let us first state an elementary result.
\begin{lemma}
\label{lemchangevar}
Let $\phi(\theta)=-m\cos (\theta-\theta_0)$ for $(m,\theta_0)\in \RR_+\times \TT$. Then, for all $f\in L^1_+(\TT\times \RR)$, we have
\begin{equation}
\label{dang}
\iint\left(\frac{v^2}{2}+\phi(\theta) \right)f^{*\phi}(\theta,v) d\theta dv=\int_{0}^{+\infty}f^\sharp(s) a_{\phi}^{-1}(s) ds.
\end{equation}
\end{lemma}
\begin{proof}[Proof of Lemma \ref{lemchangevar}]
By a first change of variable with respect to $v$: $e=\frac{v^2}{2}+\phi(\theta)$, we get
\begin{align*}
\iint\left(\frac{v^2}{2}+\phi \right)f^{*\phi} d\theta dv
&=\sqrt{2}\int_0^{2\pi}\int_{\phi(\theta)}^{+\infty}f^\sharp\circ a_{\phi}(e) e(e-\phi(\theta))^{-1/2}de d\theta\\&=\sqrt{2}\int_{-m}^{+\infty}\int_{\phi(\theta)<e}f^\sharp\circ a_{\phi}(e) e(e-\phi(\theta))^{-1/2} d\theta de.
\end{align*}
Now, if $m>0$, we deduce from Lemma \ref{lemtechnique} that $e\mapsto a_\phi(e)=\sqrt{m}\alpha_1(e/m)$ is a strictly increasing $\mathcal C^1$ diffeomorphims from $E_m=(-m,m)\cup(m,+\infty)$ onto $\RR_+^*$. Moreover, from \eqref{aphi'}, we get
$$a_\phi'(e)=\sqrt{2}\int_{\phi(\theta)<e}(e-\phi(\theta))^{-1/2}d\theta,$$
and
$$\iint\left(\frac{v^2}{2}+\phi \right)f^{*\phi} d\theta dv
=\int_{e\in E_m}f^\sharp\circ a_{\phi}(e) e a_{\phi}'(e)de,$$ 
so, performing the change of variable $s=a_\phi(e)$ on $E_m$, we obtain \eqref{dang}. If $m=0$, we observe that $a_\phi(e)=4\pi \sqrt{2e}$, $a_\phi^{-1}(s)=\frac{s^2}{32\pi^2}$ and
\begin{align*}
\iint\left(\frac{v^2}{2}+\phi \right)f^{*\phi} d\theta dv
&=2\pi\sqrt{2}\int_{0}^{+\infty}f^\sharp\left(4\pi \sqrt{2e}\right) \sqrt{e}de=\int_{0}^{+\infty}f^\sharp(s)\frac{s^2}{32\pi^2}ds.
\end{align*}
The proof of the lemma is complete.
\end{proof}
{}From \eqref{dang}, we deduce
\begin{align*}
I_2&=\iint\left(\frac{v^2}{2}+\phi_f \right)\left(f^{*\phi_f}-f_0^{*\phi_f}\right) d\theta dv\\
&=\int_{0}^{+\infty}\left(f^\sharp(s)-f_0^\sharp(s)\right) a_{\phi_f}^{-1}(s) ds\\
&=\int_{0}^{+\infty}\left(f^\sharp(s)-f_0^\sharp(s)\right) \left(a_{\phi_f}^{-1}(s)-\min \phi_f\right) ds+\min \phi_f\int_{0}^{+\infty}\left(f^\sharp(s)-f_0^\sharp(s)\right)  ds\\
&\geq \int_{f^\sharp(s)<f_0^\sharp(s)}\left(a_{\phi_f}^{-1}(s)-\min \phi_f\right)\left(f^\sharp(s)-f_0^\sharp(s)\right)  ds-\|\phi_f\|_{L^\infty}\|f^\sharp-f_0^\sharp\|_{L^1}\\
&\geq -\int_0^{+\infty} a_{\phi_f}^{-1}(s) 
\left(f_0^\sharp(s)-f^\sharp(s)\right)_+ ds-2\|\phi_f\|_{L^\infty}\|f^\sharp-f_0^\sharp\|_{L^1},
\end{align*}
where we used that $a_{\phi_f}^{-1}(s)\geq \min \phi_f$ from \eqref{encadrement}. 
%Next, we use \eqref{equivalence} and the monotonicity of $a_{\phi_f}^{-1}$ to get
Again from \eqref{encadrement}, we obtain
$$a_{\phi_f}^{-1}(s)\leq a_{\phi_{f_0}}^{-1}(s)-\min \phi_0+\max\phi_f\leq
 a_{\phi_{0}}^{-1}(s)-\min \phi_0+\|\phi_f\|_{L^\infty}$$
and thus
\begin{align*}
I_2&\geq -\int_{0}^{+\infty}(a_{\phi_{0}}^{-1}(s)-\min \phi_0)\left(f_0^\sharp(s)-f^\sharp(s)\right)_+  ds-3\|\phi_f\|_{L^\infty}\|f^\sharp-f_0^\sharp\|_{L^1}.
\end{align*}
Using \fref{encadrement}, we have
\begin{align*}
I_2&\geq -\int_{0}^{+\infty}s^2\left(f_0^\sharp(s)-f^\sharp(s)\right)_+ds- (2m_0+3\|\phi_f\|_{L^\infty})\|f^\sharp-f_0^\sharp\|_{L^1}.
\end{align*}

%Now, we write
%\begin{align*}
%\mu_{f_0}(t)&=\left\thickvert\left\{(\theta,v)\in\TT\times \RR\,:\, F\left(\frac{v^2}{2}+\phi_0(\theta)\right)>t\right\}\right\thickvert\\
%&=\left\thickvert\left\{(\theta,v)\in\TT\times \RR\,:\, \frac{v^2}{2}+\phi_0(\theta)<F^{-1}(t)\right\}\right\thickvert\\
%&=a_{\phi_0}\left(F^{-1}(t)\right)
%\end{align*}
%This implies that $\textcolor{red} {- a_{\phi_0}^{-1}(\mu_{f_0}(t))= \min( -F^{-1}(t), m_0)}=m_0 -\left(F^{-1}(t)+m_0\right)_+ 
%%\leq \left(F^{-1}(t)+m_0\right)_+  
%,$
%and then
%\begin{align*}
%I_2&\geq - \int_{0}^{+\infty} \textcolor{red} {\left(F^{-1}(f^\sharp(s))+m_0\right)_+ }
%%\left|F^{-1}(f^\sharp(s))\right|
% \left| f_0^\sharp(s)-f^\sharp(s)\right| ds-\left(\|\phi_0\|_{L^\infty}+3\|\phi_f\|_{L^\infty}\right)\|f^\sharp-f_0^\sharp\|_{L^1}.
%\end{align*}
We now conclude by observing that, for all $\theta\in \TT$, we have
\begin{equation}
\label{phiLinfini}
|\phi_f(\theta)|\leq |M_f|\,|u(\theta)|=|M_f|\leq \|f\|_{L^1}.
\end{equation}
%and the same holds for $\phi_0$.
\end{proof}

\section{Study of the functional $\mathcal J$.}
\label{section3}
In this section, we study the function $\mathcal J(m)$ defined for $m\in\RR_+$ by \eqref{defJ}, with $\phi(\theta)=-m\cos \theta$. For $e\in \RR$ and $m\in \RR_+^*$, we recall that 
$$a_\phi(e)=\alpha_{m}(e)=\sqrt{m}\,\alpha_1\left(\frac{e}{m}\right),$$
where $\alpha_1$ was defined by \eqref{alpha1}. Clearly, \eqref{defJ} and \eqref{dang} yield, for $m>0$,
\begin{align*}
\calJ(m)&=\frac{m^2}{2}+\int_{-\infty}^{+\infty} \int^{2 \pi}_{0}\left(\frac{v^2}{2}-m\cos \theta \right)f_0^\sharp \circ \alpha_m \left(\frac{v^2}{2}-m\cos \theta \right)d\theta dv\\
&=\frac{m^2}{2}+m\int_{0}^{+\infty}f^\sharp_0(s) \,\alpha_1^{-1}\left(\frac{s}{\sqrt{m}}\right) ds.
\end{align*}
\begin{proposition}
\label{propJ}
The function $\mathcal J$ defined by \eqref{defJ} is a $\mathcal C^2$ function on $\RR_+$. Denoting $\phi(\theta)=-m\cos \theta$, we have
\begin{equation}
\label{Jprime}
\mathcal J'(m)=m-\iint f_0^{*\phi}(\theta,v)\cos \theta\,d\theta dv
\end{equation}
and
\begin{equation}
\label{Jseconde}
\mathcal J''(m)=1+\iint(f^\sharp_0\circ a_\phi)'(e(\theta,v))\left(\cos \theta -\frac{\displaystyle\int_0^{\theta_m(e(\theta,v))}\cos \theta'(e(\theta,v)+m \cos \theta')^{-1/2}\,d\theta'}{\displaystyle\int_0^{\theta_m(e(\theta,v))}(e(\theta,v)+m \cos \theta')^{-1/2}\,d\theta'}\right)^2\,d\theta dv,
\end{equation}
where $e(\theta,v)=\frac{v^2}{2}+\phi(\theta)$ and $\theta_m$ is defined by \eqref{defthetam}.
\end{proposition}
{}From this Proposition and from \eqref{fixedpoint}, it is immediate to deduce the
\begin{corollary}
\label{coroJ}
Under Assumption \ref{criterion}, the magnetization $m_0$ of the stationary state $f_0$ is a strict local minimizer of $\mathcal J$: one has
$$\mathcal J'(m_0)=0\qquad \mbox{and}\qquad \mathcal J''(m_0)=1-\kappa_0>0.$$
\end{corollary}
\begin{proof}[Proof of Proposition \ref{propJ}]
To differentiate the function $\mathcal J(m)$, we denote
$$g(m,s)=f^\sharp_0(s) \,\alpha_1^{-1}\left(\frac{s}{\sqrt{m}}\right).$$
{}From Lemma \ref{lemtechnique}, $g$ is continuously differentiable with respect to $m\in \RR_+^*$, with
$$\frac{\pa g}{\pa m}(m,s)=-\frac{sf^\sharp_0(s)}{2m^{3/2}\,\alpha_1'\circ \alpha_1^{-1}(\frac{s}{\sqrt{m}})}.$$
Moreover, we can also easily deduce from Lemma \ref{lemtechnique} that there exists a constant $C>0$ such that
\begin{equation}
\label{minoraphi'}
\sqrt{2+e}\,\alpha_1'(e)\geq C, \qquad \forall e\geq -1.
\end{equation}
Let us fix $0<m_1<m_2$. We deduce from \eqref{minoraphi'} that, for all $(m,s)\in [m_1,m_2]\times \RR_+$,
$$\left|\frac{\pa g}{\pa m}(m,s)\right|\lesssim sf^\sharp_0(s)\left(2+\alpha_1^{-1}\left(\frac{s}{\sqrt{m}}\right)\right)^{1/2},$$
where $f\lesssim g$ means $f\leq Cg$ for some constant $C$. Next, using \eqref{encadrement}, we obtain
$$\left|\frac{\pa g}{\pa m}(m,s)\right|\lesssim (1+s^2) f^\sharp_0(s).$$
Now, we claim that
\begin{equation}
\label{fsharpenergy}
\int_0^{+\infty}(1+s^2) f^\sharp_0(s)ds<+\infty.
\end{equation}
Indeed, we already know that $\int f^\sharp_0(s)ds=\|f_0\|_{L^1}<+\infty$ and, by \eqref{encadrement} and \eqref{dang},
\begin{align*}
\int_0^{+\infty}s^2f^\sharp_0(s)ds
&\lesssim \int_0^{+\infty}\left(1+a_{\phi_0}^{-1}(s)\right)f^\sharp_0(s)ds\\
&\qquad =\iint\left(1+\frac{v^2}{2}+\phi_0(\theta) \right)f_0^{*\phi_0}(\theta,v) d\theta dv\\
&\lesssim \iint\left(1+v^2+\|f_0\|_{L^1}\right)f_0(\theta,v) d\theta dv<+\infty,
\end{align*}
where we used \eqref{fixedpoint}, \eqref{phiLinfini} and Assumption \ref{hypF}. This proves \eqref{fsharpenergy} and, by dominated convergence, one can continuously differentiate $\mathcal J(m)=\frac{m^2}{2}+m\int_0^{+\infty}g(m,s)ds$ for all $m>0$:
\begin{align*}
\mathcal J'(m)&=m+\int_0^{+\infty}g(m,s)ds+m\int_0^{+\infty}\frac{\pa g}{\pa m}(m,s)ds\\
&=m+\int_0^{+\infty}f^\sharp_0(s)\left(\alpha_1^{-1}\left(\frac{s}{\sqrt{m}}\right)-\frac{s}{2\sqrt{m}\,\alpha_1'\circ \alpha_1^{-1}(\frac{s}{\sqrt{m}})}\right)ds\\
&=m+\frac{1}{m}\int_0^{+\infty}f^\sharp_0(s)\left(a_\phi^{-1}(s)-\frac{s}{2 a_\phi'\circ a_\phi^{-1}(s)}\right)ds\\
&=m+\frac{1}{m}\int_{-m}^{+\infty}f^\sharp_0\circ a_\phi(e)\left(ea_\phi'(e)-\frac12 a_\phi(e)\right)de.
\end{align*}
We now introduce the function
$$\beta_1(e)=2\sqrt{2}\int_0^{2\pi}\cos \theta\sqrt{(e+\cos \theta)_+}\,d\theta \qquad \mbox{for }e\in \RR$$
and denote 
\begin{equation}
\label{defbphi}
b_\phi(e)=2\sqrt{2}\int_0^{2\pi}\cos \theta\sqrt{(e+m \cos \theta)_+}\,d\theta=\sqrt{m}\,\beta_1\left(\frac{e}{m}\right).
\end{equation}
Let us list a few properties of this function $b_\phi$. By adapting the proof of Lemma \ref{lemtechnique} developed in the Appendix, it is readily seen that $b_\phi$ is a continuous function on $\RR$, vanishing for $e\leq -m$, continuously differentiable on $[-m,m)\cup(m,+\infty)$ with
$$b_\phi'(e)=2\sqrt{2}\int_0^{\theta_m(e)}\cos \theta(e+m\cos \theta)^{-1/2}\,d\theta.$$
Moreover, we have
\beq
b_\phi(e)=4\sqrt{2}\int_0^{\pi/2}\cos\theta\left(\sqrt{(e+m \cos \theta)_+}-\sqrt{(e-m \cos \theta)_+}\right)\,d\theta
\nonumber\\
%&=8m\sqrt{2}\int_0^{\theta_m(e)}\frac{(\cos\theta)^2}{\sqrt{e+m \cos \theta}+\sqrt{(e-m \cos \theta)_+}}\,d\theta,
\label{defbphitilde}
\eeq
which implies that $b_\phi(e)$ is always positive for $e>-m$, $m>0$.  
For $e>m$ we then have
\beq
b_\phi(e)=8m\sqrt{2}\int_0^{\pi/2}\frac{(\cos\theta)^2}{\sqrt{e+m \cos \theta}+\sqrt{(e-m \cos \theta)_+}}\,d\theta,
\eeq
\begin{equation}
\label{equivb}
b_\phi(e)\sim \frac{\pi m\sqrt{2}}{\sqrt{e}}\qquad \mbox{as }e\to +\infty.
\end{equation}
Similarly, for $e>m$ we have
$$
b_\phi'(e)=- 4m\sqrt{2}\int_0^{\pi/2}\frac{(\cos\theta)^2}{\sqrt{e+m \cos \theta}\sqrt{e-m \cos \theta}(\sqrt{e+m \cos \theta}+\sqrt{e-m \cos \theta}}\,d\theta,
$$
thus
\begin{equation}
\label{equivb'}
 b_\phi'(e)\sim -\frac{\pi m }{e\sqrt{2e}}\qquad \mbox{as }e\to +\infty.
\end{equation}
Now we observe that
\begin{equation}
\label{lien ab}ea_\phi'(e)+mb_\phi'(e)=\frac12 a_\phi(e).
\end{equation}
Hence, for $m>0$, we have
\begin{align}
\label{J'}
\mathcal J'(m)&=m-\int_{-m}^{+\infty}f^\sharp_0\circ a_\phi(e)b_\phi'(e)\, de\\
&=m-2\sqrt{2}\int_{-m}^{+\infty}\int_0^{\theta_m(e)}f^\sharp_0\circ a_\phi(e)\frac{\cos \theta}{\sqrt{e+m\cos \theta}}\,d\theta de.\nonumber
\end{align}
By passing to the limit in this formula, we also get that $\mathcal J$ is differentiable at $m=0$, with $\mathcal J'(0)=0$. Finally, coming back to the variables $(\theta,v)$, we obtain \eqref{Jprime}.

\bs
In order to compute the second derivative of $\mathcal J$, let us transform this expression into a more suitable one, using an integration by parts in $e$. We denote $\widetilde e_*=a_{\phi}^{-1}\circ a_{\phi_0}(e_*)$, where $e_*$ is defined in Assumption \ref{hypF}. By \eqref{fixedpoint}, we have $f^\sharp_0\circ a_\phi=F\circ a_{\phi_0}^{-1}\circ a_\phi$, this function being continuous on $[-m,+\infty)$, of class $\mathcal C^1$ on $[-m,+\infty)\setminus\{m,\widetilde e_*\}$, nonincreasing, and vanishes on $[\widetilde e_*,+\infty)$. Therefore, in the case $e_*<+\infty$, one can directly integrate by parts to obtain
\begin{equation}
\label{IPP}\int_{-m}^{+\infty}f^\sharp_0\circ a_\phi(e)b_\phi'(e)\,de=-\int_{-m}^{+\infty}(f^\sharp_0)'\circ a_\phi(e)a_\phi'(e)b_\phi(e)\,de.
\end{equation} 

Now we deal with the case $\widetilde e_*=e_*=+\infty$. Since $f^\sharp$ is a nonincreasing function on $\RR^+$ and belongs to $L^1(\RR^+)$, we deduce that $f^\sharp (s)\to 0$ when $e\to +\infty$.  Therefore, according to \eqref{equivb}, we have $f^\sharp_0\circ a_\phi(e)b_\phi (e)\to 0$ when $e\to +\infty$, and  the integration by parts giving \eqref{IPP} is also valid in the case $e_*=+\infty$.
%we shall use Lemma \ref{lemm_ipp} proved in the Appendix, applied to the functions
%$$f(x)=f^\sharp_0\circ a_\phi(e+m),\qquad g(x)=b_\phi(e+m).$$
%We have seen that $f$ is continuous and decreasing on $\RR_+$, piecewise $\mathcal C^1(\RR_+^*)$, that $g\geq 0$ is continuous on $\RR_+$,  also piecewise $\mathcal C^1(\RR_+^*)$ and positive on $\RR_+^*$. Moreover, by \eqref{equivb} and \eqref{equivb'}, the function $g'(x)/g(x)\sim_{+\infty}1/(2x)$ is not integrable on $[1,+\infty)$. Let us prove that $fg'$ belongs to $L^1(\RR_+)$. Remarking that
%\begin{equation}
%\label{comparaison}
%|b_\phi'(e)|\leq a_\phi'(e),
%\end{equation}
%we get
%\begin{align*}
%\int_{0}^{+\infty}\left|f(x)g'(x)\right|dx&=\int_{-m}^\infty f^\sharp_0(a_\phi(e))\left|b_\phi'(e)\right|de\\
%&\leq \int_{-m}^\infty f^\sharp_0(a_\phi(e))a_\phi'(e)de= \int_{0}^\infty f^\sharp_0(s)ds=\|f_0\|_{L^1}.
%\end{align*}
%The function $fg'$ thus belongs to $L^1(\RR_+)$ and $f$ and $g$ satisfy the assumptions of Lemma \ref{lemm_ipp}: the integration by parts giving \eqref{IPP} is also valid in the case $e_*=+\infty$.

Consequently, we have
\begin{align*}
\mathcal J'(m)&=m+\int_{-m}^{+\infty}(f^\sharp_0)'\circ a_\phi(e)a_\phi'(e)b_\phi(e) de\\
&=m+\int_0^{+\infty}(f^\sharp_0)'(s)b_\phi\circ a_\phi^{-1}(s) ds\\
&=m+\sqrt{m}\int_0^{+\infty}(f^\sharp_0)'(s)\,\beta_1\circ \alpha_1^{-1}\left(\frac{s}{\sqrt{m}}\right)\,ds.
\end{align*}
Consider the function
$$h(m,s)=(f^\sharp_0)'(s) \,\beta_1\circ \alpha_1^{-1}\left(\frac{s}{\sqrt{m}}\right).$$
Using again Lemma \ref{lemtechnique}, we get that $h$ is continuously differentiable with respect to $m\in \RR_+^*$ for all $m\in \RR_+^*\backslash\{s^2/32\}$, with
$$\frac{\pa h}{\pa m}(m,s)=-\frac{s(f^\sharp_0)'(s)}{2m^{3/2}\,\alpha_1'\circ \alpha_1^{-1}(\frac{s}{\sqrt{m}})}\,\beta_1'\circ\alpha_1^{-1}\left(\frac{s}{\sqrt{m}}\right).$$
Since $|b_\phi'(e)|\leq a_\phi'(e)$,  we deduce that
$$\left|\frac{\pa h}{\pa m}(m,s)\right|\lesssim -s(f^\sharp_0)'(s), \quad \mbox{for all} \ m\in[m_1,m_2], \ 0<m_1<m_2.$$
We now claim that 
\beq
\label{sfsharp}
\mbox{the fonction}\ s\mapsto s(f_0^\sharp)'(s)\ \mbox{belongs to}\  L^1(\RR_+).
\eeq
Indeed, since $f_0^\sharp$ is decreasing, we have
$$\int_0^{r} s^2 f^\sharp_0(s) ds\geq  f^\sharp_0(r)\int_0^{r} s^2ds= \frac{r^3}{3}  f_0^\sharp(r).$$
Hence, using \fref{fsharpenergy}, we get
$$ f^\sharp_0(s)\lesssim \frac{1}{s^3}, \quad \forall s>0.$$
In particular $s f^\sharp_0(s)\to 0$ when $s\to +\infty$.
On the other hand, the function $f^\sharp_0=F\circ a_{\phi_0}^{-1}$ is continuous on $\RR_+$, of class $\mathcal C^1$ and decreasing on $[0,a_{\phi_0}(e_*))$, vanishing on $[a_{\phi_0}(e_*),+\infty)$ (with possibly $a_{\phi_0}(e_*)=+\infty$). Therefore we can perform the following integration by parts
$$- \int_0^{+\infty}  s(f_0^\sharp)'(s)ds= \int_0^{+\infty} f_0^\sharp(s)ds <+\infty.$$
This ends the proof of claim \fref{sfsharp} and 
% Moreover, from Lemma \ref{lemm_ipp} applied to the functions $f(x)=f_0^\sharp(x)$ and $g(x)=x$, we deduce that the function $s(f_0^\sharp)'(s)$ belongs to $L^1(\RR_+)$. 
enables to conclude by dominated convergence that $\mathcal J'$ is continuously differentiable on $\RR_+$ and that
\begin{align*}
&\mathcal J''(m)=1+\frac{1}{2\sqrt{m}}\int_0^{+\infty}h(m,s)ds+\sqrt{m}\int_0^{+\infty}\frac{\pa h}{\pa m}(m,s)ds\\
&\qquad =1+\int_0^{+\infty}(f^\sharp_0)'(s)\left(\frac{1}{2\sqrt{m}}\beta_1\circ \alpha_1^{-1}\left(\frac{s}{\sqrt{m}}\right)-\frac{s\beta_1'\circ\alpha_1^{-1}\left(\frac{s}{\sqrt{m}}\right)}{2m\,\alpha_1'\circ \alpha_1^{-1}(\frac{s}{\sqrt{m}})}\right)ds\\
&\qquad =1+\frac{1}{2m}\int_{-m}^{+\infty}\frac{(f^\sharp_0\circ a_\phi)'(e)}{a_\phi'(e)}\left(a_\phi'(e)b_\phi(e)-a_\phi(e)b_\phi'(e)\right)de.
\end{align*}
Finally, observing from \eqref{lien ab} and from
$$eb_\phi'(e)+2m\sqrt{2}\int_0^{\theta_m(e)}(\cos\theta)^2(e+m \cos \theta)^{-1/2}\,d\theta=\frac12 b_\phi(e)$$
that
\begin{align*}
&\frac{1}{2ma_\phi'(e)}\left(a_\phi'(e)b_\phi(e)-a_\phi(e)b_\phi'(e)\right)\\
&\qquad=-\frac{(b_\phi'(e))^2}{a_\phi'(e)}+2\sqrt{2}\int_0^{\theta_m(e)}(\cos\theta)^2(e+m \cos \theta)^{-1/2}\,d\theta\\
&\qquad =2\sqrt{2}\int_0^{\theta_m(e)}\left(\cos \theta -\frac{\ds\int_0^{\theta_m(e)}\cos \theta'(e+m \cos \theta')^{-1/2}\,d\theta'}{\ds \int_0^{\theta_m(e)}(e+m \cos \theta')^{-1/2}\,d\theta'}\right)^2(e+m \cos \theta)^{-1/2}\,d\theta,
\end{align*}
we obtain \eqref{Jseconde} by coming back to the $(\theta,v)$ variables.
\end{proof}

\section{Control of $f$}
\label{section4}

Our previous analysis has allowed the control of the magnetization vector by the relative Hamiltonian and 
the relative rearrangements. It remains to control the whole distribution function $f$. To this aim we now 
write the relative energy in the following form:

\begin{equation}
\label{eq1234}
\ham(f)-\ham(f_0)= \iint\left(\frac{v^2}{2}+\phi_{f_0}\right)\left(f-f_0\right) d\theta dv-\frac{1}{2}|M_f-M_{f_0}|^2.
\end{equation}
In particular, this means that  the following quantity 
$$ \iint\left(\frac{v^2}{2}+\phi_{f_0}\right)\left(f-f_0\right) d\theta dv$$
is controlled  and the  problem is to show how this quantity controls $f-f_0$.  This task was achieved in the context of the gravitational Vlasov-Poisson system \cite{ML} using compactness arguments.  Here we will rather use a functional inequality established in \cite{Lq} to get a quantitative control of $\|f-f_0\|_{L^1}$ by this quantity, up to rearrangement terms depending only on $f^*$ and $f_0^*$ which are preserved by the flow.  We emphasize that the steady states to Vlasov-Poisson system studied in  \cite{ML}  are compactly supported and this property was essential to  succcessfully drive the stability analysis in this context. Here this assumption is not needed and a much weaker assumption is made in the case of the HMF model. More precisely, we have the following inequality:

\begin{proposition}
\label{prop-ineq-quant}
Let $f_0$ be given by \fref{Qexp} where $F$ satisfies Assumption \ref{hypF}. Then, there exist a constant $K_0$  depending only on $f_0$ such that, for all $f\in L^1((1+|v|^2) dv d\theta)$ we have
\begin{align}
\label{ineq-quant}
\left(\| f- f_0\|_{L^1} + \|f_0\|_{L^1} -\|f\|_{L^1}\right)^2\leq & K_0\iint\left(\frac{v^2}{2}+\phi_{f_0}\right)\left(f-f_0\right) d\theta dv \nonumber \\ &
\hspace{-1cm}+ m_0\|f^*-f_0^*\|_{L^1} + \frac{1}{8\pi^2} \int_0^{+\infty}\mu_0(s)^2\beta_{f^*,f_0^*}(s) ds,
 \end{align}
 where  $\beta_{f^*,f_0^*}(s) =\left\thickvert\left\{(\theta,v) \in \TT\times \Real: f^*(\theta,v)\leq s<f_0^*(\theta, v)\right\}\right\thickvert, \mbox{for all}\ s\geq 0.$
\end{proposition}
\begin{proof}
We shall apply Theorem 1 in \cite{Lq}. We use the rearrangement with respect to $e_0(\theta,v)= \frac{v^2}{2}+\phi_{f_0}$ and  recall that the function $a_{\phi_0}$ is strictly increasing and a one-to-one  function  from $[\min \phi_0, +\infty)$
to $[0, +\infty)$. Following \cite{Lq}, we introduce the functions 
\begin{equation}
B_0(\mu)= \int_0^{\mu} a_{\phi_0}^{-1}(s) ds, \quad \forall  \mu\geq 0,
\end{equation}
and
$$
 H_0(\mu) = \inf_{0< s\leq \mu}   \frac{B_0(\mu+s)+B_0(\mu-s)-2B_0(\mu)}{s^2}.
$$
Then from Theorem 1 in \cite{Lq}, we have
\begin{align}\left(\| f- f_0\|_{L^1} + \|f_0\|_{L^1} -\|f\|_{L^1}\right)^2\leq & K(f_0) \iint\left(\frac{v^2}{2}+\phi_{f_0}\right)\left(f-f_0\right) d\theta dv \nonumber  \\
&\hspace{-3cm}+  \int_0^{+\infty}  a_{\phi_0}^{-1} (2\mu_{f_0}(s))\beta_{f^*,f_0^*}(s)ds - \int_0^{+\infty}a_{\phi_0}^{-1} (\mu_{f_0}(s))\beta_{f^*_0,f^*}(s)ds \label{inequ123} 
\end{align}
where
\beq
\label{Kf0} K(f_0)= 4 \int_0^{\|f_0\|_{L^\infty}} \frac{ds}{H_0(\mu_{f_0}(s))}\ , \ \mbox{and} 
\eeq
\beq
\label{nut}
\beta_{f,g}(s)= \mbox{meas}\{(\theta,v)\in \TT \times \RR; f(\theta, v)\leq s<g(\theta,v)     \}, \quad \forall s\geq 0.
\eeq
Using the estimates \fref{encadrement} we then get from \fref{inequ123}
\begin{align}\left(\| f- f_0\|_{L^1} + \|f_0\|_{L^1} -\|f\|_{L^1}\right)^2\leq & K(f_0) \iint\left(\frac{v^2}{2}+\phi_{f_0}\right)\left(f-f_0\right) d\theta dv \nonumber  \\
&\hspace{-4cm}+ \frac{1}{8\pi^2} \int_0^{+\infty}\mu_0(s)^2\beta_{f^*,f_0^*}(s) ds +m_0\int_0^{+\infty} \left(\beta_{f^*_0,f^*}(s)+ \beta_{f^*,f_0^*}(s)\right)ds \label{inequ4} 
\end{align}
Observing that 
$$\int_0^{+\infty} \beta_{f,g}(s)ds = \iint (g-f)_+d\theta dv,
$$
we get
$$ \int_0^{+\infty} \left(\beta_{f^*_0,f^*}(s)+ \beta_{f^*,f_0^*}(s)\right)ds= \|f^*-f_0^*\|_{L^1}, $$
and therefore
\begin{align}\left(\| f- f_0\|_{L^1} + \|f_0\|_{L^1} -\|f\|_{L^1}\right)^2\leq & K(f_0) \iint\left(\frac{v^2}{2}+\phi_{f_0}\right)\left(f-f_0\right) d\theta dv \nonumber  \\
&\hspace{-4cm}+  \frac{1}{8\pi^2} \int_0^{+\infty}\mu_0(s)^2\beta_{f^*,f_0^*}(s) ds+ m_0\|f^*-f_0^*\|_{L^1}. \label{inequ5} 
\end{align}
 To end the proof of inequality \fref{ineq-quant}, it only remains to show that the  quantity  $ K(f_0)$ is finite. 
First we rewrite $H_0(\mu)$ as
\begin{align*}
H_0(\mu)&=\inf_{0< s\leq \mu}   \frac{B_0(\mu+s)+B_0(\mu-s)-2B_0(\mu)}{s^2}\\
&= \inf_{0< s\leq \mu} \int_0^1(1-\lambda) ( (a_{\phi_0}^{-1})'(\mu +\lambda s) +(a_{\phi_0}^{-1})'(\mu -\lambda s))d\lambda.\\
\end{align*}
Then, from the properties  of $a_\phi$ stated in Lemma \ref{lemtechnique}, we claim that 
$$(a_{\phi_0}^{-1})'(\mu +\lambda s) +  (a_{\phi_0}^{-1})'(\mu -\lambda s)\geq (a_{\phi_0}^{-1})'(\mu),$$
for all $0\leq \lambda\leq 1$, $0< s\leq \mu$. Indeed, if $\mu \leq 16\sqrt{m_0}$ then  $\mu-\lambda s  \leq 16\sqrt{m_0}$, and since the function $a_{\phi_0}^{-1}$ is a concave function on $[0,16\sqrt{m_0}]$, we have $(a_{\phi_0}^{-1})'(\mu-\lambda s)\geq (a_{\phi_0}^{-1})'(\mu)$. Therefore we have the desired claim in this case. Similarly, if $\mu \geq 16\sqrt{m_0}$ then  $\mu+\lambda s  \geq 16\sqrt{m_0}$, and since the function $a_{\phi_0}^{-1}$ is a convex function on $[16\sqrt{m_0}, + \infty]$, we have $(a_{\phi_0}^{-1})'(\mu+\lambda s)\geq (a_{\phi_0}^{-1})'(\mu)$. Therefore the above claim holds in both cases. Using this claim, we then get
$$
H_0(\mu) \geq  \frac{1}{2a_{\phi_0}'\circ a_{\phi_0}(\mu)},$$
 and thus
$$ K(f_0)\leq 8 \int_0^{\|f_0\|_{L^\infty}} a_{\phi_0}'\circ a_{\phi_0}^{-1}(\mu_{f_0}(t))dt.$$
Now we observe that   forl all $t\geq 0$
$$\mu_{f_0}(t)= \left\thickvert\left\{ F\left(\frac{|v|^2}{2}+\phi_0(x)\right)>t\right\} \right\thickvert =  a_{\phi_0}\left(F^{-1}(t) \right),$$
and therefore
$$  K(f_0)\leq 8 \int_0^{\|f_0\|_{L^\infty}} a_{\phi_0}'\left(F^{-1}(t)\right) dt.$$
We then perform the change of variable  $e=F^{-1}(t)$ to get
\beq
\label{rhsintegral}
  K(f_0)\leq 8 \int_{-m_0}^{e_*} a_{\phi_0}'(e)| F'(e)| de.
\eeq
Now we claim that  the rhs integral in this inequality is finite.  Indeed, assume first that $e_*<+\infty$. The only possible singularities in this 
integral are at $e=m_0$ and $e=e_*$, since the function $e\mapsto a_{\phi_0}'(e)| F'(e)|$ is continuous on $[-m_0,+\infty)\backslash\{m_0,e_*\}$. 

If we suppose that $e^*\neq m_0$, then we have $a_{\phi_0}'(e)|F'(e)|\sim a_{\phi_0}'(e_*)|F'(e)|$ when $e\to e_*$ and, from Lemma \ref{lemtechnique} we have (for $m_0<e_*$ otherwise $F$ vanishes in the neighborhood of $m_0$) $a_{\phi_0}'(e)| F'(e)|\sim C\log |e-m_0|$ when $e\to m_0$.
These two possible singularities are thus integrable (the first is integrable by assumption on $F$).

If $e^*=m_0$ then our Assumption \ref{hypF} {\em (i)} ensures that $\int_{-m_0}^{e_*} a_{\phi_0}'(e)| F'(e)| de$ is finite, since $a_{\phi_0}'(e)\sim C\log |e-m_0|$ as $e\to m_0$.

Assume now that $e_*=+\infty$. Using assertion {\em (iii)} of Lemma \ref{lemtechnique}, we have $a_{\phi_0}'(e)|F'(e)|\sim C|F'(e)|/\sqrt{e} $ as $e\to +\infty,$ where $C$ is a constant.  But, as for \fref{dang},  we have

\begin{align*} 
\int\int  F\left(\frac{|v|^2}{2}+\phi_0(\theta)\right) d\theta dv &=\int\int
\left(\int_0^{F\left(\frac{|v|^2}{2}+\phi_0(\theta)\right)} dt \right) d\theta dv\\
&= \int_{0}^{+\infty} \mbox{meas}\left\{F\left(\frac{|v|^2}{2}+\phi_0(\theta)\right)>t  \right\} dt\\
&=  \int_{0}^{+\infty} a_{\phi_0}(F^{-1}(t)) dt = \int_{-m_0}^{+\infty} a_{\phi_0}(e) |F'(e)| de.
\end{align*}
This implies that the integral $\int_{-m_0}^{+\infty} a_{\phi_0}(e) |F'(e)| de$  is convergent. {}From assertion  {\em (iii)}, we know that
$a_{\phi_0}'(e) |F'(e)|\sim C_1|F'(e)|/\sqrt{e} \leq  C_1 |F'(e)|\sqrt{e} \sim C_2a_{\phi_0}(e) |F'(e)| $ for $e$ large enough, where $C_1$ and $ C_2$  are some positive constants.
This proves the fact that the rhs integral of \fref{rhsintegral} is finite, and ends the proof of  Proposition \ref{prop-ineq-quant}.

%Now  we have to estimate $\nu(f^*(\theta ,v))$  where $\nu$ is given by \fref{nut}. From \fref{encadrement}, we have
%$$\frac{\mu^2}{8\pi^2} -m_0 \leq a_{\phi_0}^{-1}(2\mu) \leq   \frac{\mu^2}{8\pi^2} + m_0, $$ and 
%$$-\frac{\mu^2}{32\pi^2} -m_0 \leq - a_{\phi_0}^{-1}(\mu) \leq   -\frac{\mu^2}{32\pi^2} + m_0, $$
%This implies
%$$|\nu(t)| \leq \frac{\mu_{f_0}(t)^2}{8\pi^2} + m_0.$$
%Knowing that $\mu_{f_0}(t)= a_{\phi_0}\left( F^{-1}(t)\right)$ and using \fref{encadrement0} we get
%$$  |\nu(t)| \leq    4 \left(F^{-1} (t) + m_0   \right)_+ + m_0.$$ 
%Reporting this estimate into \fref{inequ123} yields  inequality \fref{ineq-quant} and ends the proof of Proposition \ref{prop-ineq-quant}.
\end{proof}

\section{Proof of Theorem \ref{mainthm}}
\label{proof-thm}
 We first insert identity \fref{eq1234} into inequality \fref{ineq-quant} and get
\begin{align}
\label{estimate1} \left(\| f- f_0\|_{L^1} + \|f_0\|_{L^1} -\|f\|_{L^1}\right)^2\leq & K_0  \left[ \ham(f)-\ham(f_0)+\frac{1}{2}|M_f-M_{f_0}|^2\right] \nonumber \\ &
\hspace{-1cm} + m_0\|f^*-f_0^*\|_{L^1} + \frac{1}{8\pi^2} \int_0^{+\infty}\mu_0(s)^2\beta_{f^*,f_0^*}(s) ds.\end{align}
 We write  $M_f= |M_f| u(\theta_{f})$ where $u(\theta)= (\cos\theta ,\sin\theta)^T$, and denote by
 $f_0(\cdot -\theta_f)$ the fonction $f_0(\cdot -\theta_f)(\theta, v)= f_0(\theta-\theta_f, v)$.
  We then apply this inequality \fref{estimate1} to $f_0(\cdot -\theta_f)$ and get
  \begin{align*}
& \left(\| f- f_0(\cdot-\theta_f)\|_{L^1} + \|f_0\|_{L^1} -\|f\|_{L^1}\right)^2\leq K_0  \left[ \ham(f)-\ham(f_0)+\frac{1}{2}|M_f-M_{f_0(\cdot-\theta_f)}|^2\right] \nonumber \\ &
 \hspace*{5cm}+ m_0\|f^*-f_0^*\|_{L^1} + \frac{1}{8\pi^2} \int_0^{+\infty}\mu_0(s)^2\beta_{f^*,f_0^*}(s) ds.\end{align*}
 Now we observe that
 \begin{align*}M_{f_0(\cdot-\theta_f)} &= \int_0^{2\pi}  \rho_{f_0}(\theta -\theta_f) u(\theta) d\theta=
   \int_0^{2\pi}  \rho_{f_0}(\theta) u(\theta+\theta_f) d\theta\\
   &=(m_0 \cos(\theta_f), m_0 \sin(\theta_f))^T= m_0u(\theta_f).
   \end{align*}
   Therefore
\begin{align}
\label{estimate2}& \left(\| f- f_0(\cdot-\theta_f)\|_{L^1} + \|f_0\|_{L^1} -\|f\|_{L^1}\right)^2\leq K_0  \left[ \ham(f)-\ham(f_0)+\frac{1}{2}\left(|M_f|-m_0\right)^2\right]  \nonumber \\ &
\hspace*{5cm}+ m_0\|f^*-f_0^*\|_{L^1} + \frac{1}{8\pi^2} \int_0^{+\infty}\mu_0(s)^2\beta_{f^*,f_0^*}(s) ds.\end{align}
Now we use Corollary \ref{coroJ} together with the fact that $\calJ$ is a $C^2$ function to conclude that there exist $\delta>0$ and $C>0$  such that
$$\calJ(m)- \calJ(m_0) \geq C (m-m_0)^2 \quad \mbox{for all}\   m \in (m_0-\delta, m_0+\delta).$$
 Reporting this into estimate \fref{estimate2} yields
 \begin{align*}&\left(\| f- f_0(\cdot-\theta_f)\|_{L^1}+ \|f_0\|_{L^1} -\|f\|_{L^1}\right)^2\leq  K_0  \left[ \ham(f)-\ham(f_0)+\frac{1}{2C}\left(\calJ(|M_f|)-\calJ(m_0)\right)\right] \nonumber \\ &
\hspace{5cm}+ m_0\|f^*-f_0^*\|_{L^1} + \frac{1}{8\pi^2} \int_0^{+\infty}\mu_0(s)^2\beta_{f^*,f_0^*}(s) ds.
 \end{align*}
 for all $f$ such that $|M_f| \in (m_0-\delta, m_0+\delta)$.
 Now using inequality \fref{inegreduction}, we get 
  \begin{align}\left(\|f- f_0(\cdot-\theta_f)\|_{L^1}+ \|f_0\|_{L^1} -\|f\|_{L^1}\right)^2&\leq  C\left[\ham(f)-\ham(f_0) + C(1+\|f\|_{L^1}) \|f^*-f_0^*\|_{L^1}{\textcolor{white}\int} \right. \nonumber\\ &
  \left. \hspace{-1cm} +  C \int_0^{+\infty}s^2 (f_0^\sharp(s)- f^\sharp(s))_+ds+  C\int_0^{+\infty}\mu_0(s)^2\beta_{f^*,f_0^*}(s) ds\right] \label{nada}.
 \end{align}
 for some positive constant $C$ only depending on $f_0$. To end the proof of Theorem \ref{mainthm}, we observe that, from the inequality $(a+b)^2\geq \frac{1}{2} a^2-b^2$, 
 \begin{align*}\left(\|f- f_0(\cdot-\theta_f)\|_{L^1}+ \|f_0\|_{L^1} -\|f\|_{L^1}\right)^2&\geq
\frac{1}{2}\|f- f_0(\cdot-\theta_f)\|_{L^1}^2 - \left(\|f_0\|_{L^1} -\|f\|_{L^1}\right)^2\\
& \geq  \frac{1}{2}\|f- f_0(\cdot-\theta_f)\|_{L^1}^2 - \|f_0^* -f^*\|_{L^1}^2 \\
& \geq \frac{1}{2}\|f- f_0(\cdot-\theta_f)\|_{L^1}^2 - \left(\|f_0\|_{L^1}+\|f\|_{L^1}\right) \|f_0^* -f^*\|_{L^1}\\
& \geq \frac{1}{2}\|f- f_0(\cdot-\theta_f)\|_{L^1}^2 - \tilde C \left(1+\|f\|_{L^1}\right) \|f_0^* -f^*\|_{L^1}
 \end{align*}
 with $\tilde C= \max(1, \|f_0\|_{L^1}).$
 We then report this into \fref{nada} and get inequality \fref{stab-inequ}
 for all $f$ such that $|M_f| \in (m_0-\delta, m_0+\delta)$. 

Let us deduce \eqref{stab-inequ-compact} in the case where if $f_0$ is a compactly supported steady state. In this case, the support of $f_0^\sharp$ is $[0,\left|\mbox{Supp}f_0\right|]$, so 
 $$\int_0^{+\infty}s^2\left(f_0^\sharp(s)- f^\sharp(s)\right)_+ ds \leq\left|\mbox{Supp}f_0\right|^2\int_0^{+\infty}\left(f_0^\sharp(s)- f^\sharp(s)\right)_+ ds\leq \left|\mbox{Supp}f_0\right|^2\|f^*-f_0^*\|_{L^1}.$$
 Furthermore, for all $s\geq 0$, we have $\mu_{f_0}(s)\leq \left|\mbox{Supp}f_0\right|$, hence
 \begin{align*}
\int_0^{+\infty}\mu_{f_0}(s)^2\beta_{f^*,f_0^*}(s) ds\leq \left|\mbox{Supp}f_0\right|^2\int_0^{+\infty} \beta_{f^*,f_0^*}(s) ds &= \left|\mbox{Supp}f_0\right|^2\iint (f_0^*- f^*)_+ d\theta dv\\ &\leq\left|\mbox{Supp}f_0\right|^2 \|f^*-f_0^*\|_{L^1}
\end{align*}
This enables to deduce \eqref{stab-inequ-compact} from \fref{stab-inequ} and this ends the proof of Theorem \ref{mainthm}.
\qed
% \begin{align}
%\label{inegreduction}
%\calJ(|M_f|)-\calJ(|M_{f_0}|)&\leq \calH(f)-\calH(f_0)+(\|f_0\|_{L^1}+3\|f\|_{L^1})\|f^*-f_0^*\|_{L^1}\\
%&\quad {\textcolor{red} +}-\int_{0}^{+\infty}\left|F^{-1}(f^\sharp(s))\right|\,|f^\sharp(s)-f_0^\sharp(s)| ds\nonumber 
%\end{align}

%-------------------------------
\section*{Appendix}
\begin{proof}[Proof of Lemma \ref{lemtechnique}]
The proof of Item {\em (i)} is straightforward. Let us prove Item {\em (ii)}. It is already clear from \eqref{alpha1} that $\alpha_1$ is strictly increasing. In order to prove that $\alpha_1'(e)$ is given by \eqref{aphi'} for all $e\in (-1,1)$, we perform the change of variable $u=\cos \theta$ in \eqref{alpha1}:
$$\alpha_1(e)=\int_{-1}^1g(e,u)\,du,\quad \mbox{where}\quad g(e,u)=4\sqrt{2}\frac{\sqrt{(e+u)_+}}{\sqrt{1-u^2}}.$$
For $u\in (-1,1)$, we have
$$0\leq \frac{\pa g}{\pa e}\leq q_e(u)=\frac{2\sqrt{2}}{\sqrt{1-e}}\,\frac{1}{\sqrt{(e+u)(1-u)}}{\mathds 1}_{-e\leq u\leq 1}$$
and, for all $e\in (-1,1)$,
$$\int_{-1}^1q_e(u)\,du=\frac{2\sqrt{2}}{\sqrt{1-e}}\pi.$$
Hence, by Br\'ezis-Lieb's Lemma \cite{liebloss}, we have $q_e\to q_{e_0}$ in $L^1(-1,1)$, for all $e_0\in (-1,1)$, and using a generalized dominated convergence theorem as stated in \cite{ML2} (Appendix A), we deduce that $\alpha_1$ is $\mathcal C^1$ on $(-1,1)$, with
\begin{equation}
\label{aphi'2}\alpha'_1(e)=2\sqrt{2}\int_{-e}^1\frac{du}{\sqrt{(e+u)(1-u^2)}}.
\end{equation}
Performing again the change of variable $u=\cos \theta$ in \eqref{aphi'2} yields \eqref{aphi'}. Now, we perform the change of variable $t=\frac{u+e}{1-u}$ in \eqref{aphi'2} and get, for $e\in (-1,1)$,
\begin{equation}
\label{aphi'3}
\alpha'_1(e)=2\sqrt{2}\int_{0}^{+\infty}\frac{dt}{\sqrt{t(1+t)(2t+1-e)}}.
\end{equation}
{}From this expression, we clearly see that $\alpha'_1$ is strictly increasing, which yields the convexity of $\alpha_1$ on $(-1,1)$. We also deduce that the right-derivative of $\alpha_1$ at $e=-1$ is finite and its value is given by
$$2\int_{0}^{+\infty}\frac{dt}{(1+t)\sqrt{t}}=2\pi.$$
Item {\em (iii)} is an easy consequence of the following expression, valid for $e>1$:
$$\alpha_1(e)=2\sqrt{2}\int_{0}^{2\pi}\sqrt{e+\cos \theta}\,d\theta.$$
Let us now prove Item {\em (iv)}. The value $\alpha_1(1)=16$ is obtained by a direct calculation. In order to prove the equivalent \eqref{equiv}, we first consider the case $e\to 1$, $e<1$. The change of variable $s=1/t$ in \eqref{aphi'3} yields
$$\alpha'_1(e)=2\sqrt{2}\int_{0}^{+\infty}\frac{ds}{\sqrt{s(1+s)(2+(1-e)s)}}.$$
Let
$$I_1(e)=2\sqrt{2}\int_{0}^{+\infty}\frac{ds}{(1+s)\sqrt{2+(1-e)s}}.$$
{}From
\begin{equation}
\label{major}
0\leq \frac{1}{\sqrt s}-\frac{1}{\sqrt {1+s}}=\frac{1}{\sqrt{s(1+s)}(\sqrt{s}+\sqrt{1+s})}\leq \frac{1}{(1+s)\sqrt{s}},
\end{equation}
we deduce that
$$
\left|\alpha'_1(e)-I_1(e)\right|\leq 2\sqrt{2}\int_{0}^{+\infty}\frac{ds}{(1+s)^{3/2}\sqrt{s(2+(1-e)s)}}\leq 2\int_{0}^{+\infty}\frac{ds}{(1+s)^{3/2}\sqrt{s}}=C_0.
$$
A direct computation yields
\begin{align*}
I_1(e)&=-\frac{2\sqrt{2}}{\sqrt{1+e}}\log\frac{1-e}{(\sqrt{2}+\sqrt{1+e})^2}\sim -2\log (1-e)\quad \mbox{as}\quad e\to 1,\,\,e<1,
\end{align*}
thus
$$\alpha'_1(e)\sim -2\log (1-e)\quad \mbox{as}\quad e\to 1,\,\,e<1.$$
To deal with the case $e\to 1$, $e>1$, we perform for $e>1$ the change of variable $t=\frac{1-\cos \theta}{1+\cos \theta}$ in \eqref{aphi'}:
$$\alpha'_1(e)=2\sqrt{2}\int_{0}^{+\infty}\frac{dt}{\sqrt{t(1+t)(t(e-1)+e+1)}}.$$
Using again \eqref{major}, we get
\begin{align*}
\left|\alpha'_1(e)-I_2(e)\right|&\leq 2\sqrt{2}\int_{0}^{+\infty}\frac{dt}{(1+t)^{3/2}\sqrt{t(t(e-1)+e+1)}}\leq C_0
\end{align*}
where we used that $t(e-1)+e+1\geq 2$ and where we set
$$
I_2(e)=2\sqrt{2}\int_{0}^{+\infty}\frac{dt}{(1+t)\sqrt{t(e-1)+e+1)}}.
$$
Since 
$$I_2=-2\log\frac{e-1}{(\sqrt{2}+\sqrt{1+e})^2}\sim -2\log (e-1)\quad \mbox{as}\quad e\to 1,\,\,e>1,$$
we infer that
$$\alpha'_1(e)\sim -2\log (e-1)\quad \mbox{as}\quad e\to 1,\,\,e>1,$$
which end the proof of {\em (iv)}. Finally, Item {\em (v)} is a straightforward consequence of Items {\em (i), (ii), (iii), (iv)}, and the proof of Lemma \ref{lemtechnique} is complete.
\end{proof}
%---------------------------
\bs

\end{document}